\documentclass[11pt,a4paper]{article}
\usepackage{pict2e}
\usepackage{booktabs}
\usepackage{color}
\usepackage{epsfig}
\usepackage{epstopdf}
\usepackage{longtable}
\usepackage[T1]{fontenc}
\usepackage{amsmath}
\usepackage{cases}
\usepackage{geometry}
\usepackage{amsbsy,latexsym,amsfonts, epsfig, color, authblk, amssymb, graphics, bm}
\usepackage{epsf,slidesec,epic,eepic}
\usepackage{fancybox}
\usepackage{fancyhdr}
\usepackage{setspace}
\usepackage{nccmath}
\usepackage{diagbox}
\usepackage[colorlinks, citecolor=blue]{hyperref}
\newtheorem{theorem}{Theorem}[section]
\newtheorem{lemma}{Lemma}[section]
\newtheorem{algorithm}{Algorithm}[section]

\newtheorem{corollary}{Corollary}[section]
\newtheorem{assumption}{Assumption}[section]
\newtheorem{remark}{Remark}[section]

\numberwithin{equation}{section}

\newenvironment{proof}{{\noindent \bf Proof:}}{\hfill$\Box$\medskip}

\definecolor{lred}{rgb}{1,0.8,0.8}
\definecolor{lblue}{rgb}{0.8,0.8,1}
\definecolor{dred}{rgb}{0.6,0,0}
\definecolor{dblue}{rgb}{0,0,0.5}
\definecolor{dgreen}{rgb}{0,0.5,0.5}

 \title{A corrected semi-proximal ADMM for multi-block convex optimization and
 its application to DNN-SDPs}

\author{Li Shen\footnote{Department of Mathematics, South China University of Technology, Guangzhou, 510641, China (shen.li@mail.scut.edu.cn).}
 \ \ {\rm and}\ \ Shaohua Pan\footnote{Corresponding author. Department of Mathematics, South China University of Technology, Tianhe District of Guangzhou City, China
 (shhpan@scut.edu.cn).}}

 \date{December 10, 2014\\
 (Revised) February 10, 2015}

\begin{document}

 \maketitle

 \begin{abstract}
  In this paper we propose a corrected semi-proximal ADMM (alternating direction method of multipliers)
  for the general $p$-block $(p\!\ge 3)$ convex optimization problems with linear constraints,
  aiming to resolve the dilemma that almost all the existing modified versions of
  the directly extended ADMM, although with convergent guarantee, often perform substantially
  worse than the directly extended ADMM itself with no convergent guarantee.
  Specifically, in each iteration, we use the multi-block semi-proximal ADMM with step-size at least
  $1$ as the prediction step to generate a good prediction point, and then make correction
  as small as possible for the middle $(p\!-\!2)$ blocks of the prediction point. Among others,
  the step-size of the multi-block semi-proximal ADMM is adaptively determined by the infeasibility
  ratio made up by the current semi-proximal ADMM step for the one yielded by the last correction step.
  For the proposed corrected semi-proximal ADMM, we establish the global convergence results
  under a mild assumption, and apply it to the important class of doubly nonnegative semidefinite
  programming (DNN-SDP) problems with many linear equality and/or inequality constraints.
  Our extensive numerical tests show that the corrected semi-proximal ADMM is superior to
  the directly extended ADMM with step-size $\tau=1.618$ and the multi-block ADMM with
  Gaussian back substitution \cite{HTY12,HY13}. It requires the least number of iterations for
  $70\%$ test instances within the comparable computing time with that of the directly extended ADMM,
  and for about $40\%$ tested problems, its number of iterations is only $67\%$ that of
  the multi-block ADMM with Gaussian back substitution \cite{HTY12,HY13}.
  \end{abstract}



  \medskip

  \section{Introduction}

  Let $\mathbb{Z}_1,\ldots,\mathbb{Z}_p$ and $\mathbb{X}$ be the real finite dimensional vector spaces
  which are equipped with an inner product $\langle \cdot,\cdot\rangle$ and its induced norm $\|\cdot\|$.
  We consider the following multi-block convex minimization problem with linear constraints in which
  the objective function is a sum of $p\ (p\ge 3)$ closed proper convex functions without overlapping variables:
  \begin{align}\label{prob}
   & \min_{z_i\in\mathbb{Z}_i}{\textstyle\sum_{i=1}^p}\theta_i(z_i)\nonumber\\
   &\ {\rm s.t.}\ {\textstyle\sum_{i=1}^p}\mathcal{A}_i^*z_i= c,
  \end{align}
  where $\theta_i\!:\mathbb{Z}_i\to(-\infty,+\infty]$ for $i=1,\ldots,p$ are closed proper convex
  functions, $\mathcal{A}_i^*\!:\mathbb{Z}_i\to\mathbb{X}$ for $i=1,\ldots,p$ are the adjoint
  operator of $\mathcal{A}_i\!:\mathbb{X}\to\mathbb{Z}_i$, and $c\in\mathbb{X}$
  is the given vector. Throughout this paper, we assume that problem (\ref{prob}) has an optimal solution.

  \medskip

  There are many important cases that take the form of (\ref{prob}). One compelling example is
  the so-called robust PCA (principle component analysis) with noisy and incomplete data considered in
  \cite{WGRPM09,TYuan11} and the low rank matrix required to be nonnegative, which can be modelled as
  (\ref{prob}) in which the objective function is the sum of the nuclear norm, the $\ell_1$-norm,
  and the indicator function over the nonnegative orthant cone. Another prominent example comes from
  the matrix completion with or without fixed basis coefficients \cite{Gross11,MPS14,Negahban11},
  for which the nuclear norm and the nuclear semi-norm penalized least squares convex relaxation problems
  exactly have a dual of the form (\ref{prob}). Another interesting example is the simultaneous
  minimization of the nuclear norm and $\ell_1$-norm of some structured matrix,
  which arises frequently from the structured low-rank and sparse representation for image classification
  and subspace clustering \cite{ZJD13,WXL13}. In Section 4 of this paper, we focus on the solution of model
  (\ref{prob}) from the doubly nonnegative semidefinite programming (DNN-SDP) problems with many linear
  equality and/or inequality constraints, which arise in convex relaxation for some difficult combinatorial
  optimization problems.

  \medskip

  The alternating direction method of multipliers (ADMM for short) was first proposed by Glowinski
  and Marrocco \cite{GM75} and Gabay and Mercier \cite{GM76} for the convex problem
  \begin{equation}\label{prob1}
    \min_{z_1\in\mathbb{Z}_1,z_2\in\mathbb{Z}_2}\Big\{\theta_1(z_1)+\theta_2(z_2)\ |\ \mathcal{A}_1^*z_1+\mathcal{A}_2^*z_2=c\Big\}.
  \end{equation}
  Let $l_{\sigma}\!: \mathbb{Z}_1\times\mathbb{Z}_2\times\mathbb{X}\to(-\infty,+\infty]$
  be the Lagrange function of model (\ref{prob1}) defined by
  \[
   l_{\sigma}(z_1,z_2;x):=\theta_1(z_1)+\theta_2(z_2)+\langle x, \mathcal{A}_1^*z_1+\mathcal{A}_2^*z_2-c \rangle
   +\frac{\sigma}{2}\big\|\mathcal{A}_1^*z_1+\mathcal{A}_2^*z_2-c\big\|^2,
  \]
  where $\sigma>0$ is the penalty parameter.
  For a chosen initial point $(z_1^0,z_2^0,x^0)\in {\rm dom}\,\theta_1\times{\rm dom}\,\theta_2\times\mathbb{X}$,
  the ADMM consists of the following iteration steps
  \begin{subnumcases}{}
  z_1^{k+1}\in\mathop{\arg\min}_{z_1\in\mathbb{Z}_1}\ l_{\sigma}(z_1,z_2^{k};x^{k}),\\
  z_2^{k+1}\in\mathop{\arg\min}_{z_2\in\mathbb{Z}_2}\ l_{\sigma}(z_1^{k+1},z_2;x^{k}),\\
  x^{k+1}=x^{k} +\tau\sigma\big(\mathcal{A}_1^*z_1^{k+1}+\mathcal{A}_2^*z_2^{k+1}-c\big),
  \label{multiplier1}
  \end{subnumcases}
  where $\tau\in(0,\frac{1+\sqrt{5}}{2})$ is a constant to control the step-size in
  \eqref{multiplier1}. The iterative scheme of ADMM actually embeds a Gaussian-Seidel
  decomposition into each iteration of the classical augmented Lagraigan method of
  Hestenes-Powell-Rockafellar \cite{Henstenes76,Powell69,Roc76}, so that the challenging task,
  i.e. the exact solution or the approximate solution with a high precision of the Lagrangian
  minimization problem, is relaxed to several easy ones.

  \medskip

  Motivated by the same philosophy, one naturally extends the above $2$-block ADMM to
  the multi-block convex minimization problem (\ref{prob}) directly.
  Let $L_{\sigma}\!: \mathbb{Z}_1\times\cdots\times\mathbb{Z}_p\times\mathbb{X}\to(-\infty,+\infty]$
  denote the augmented Lagrange function for model \eqref{prob}, defined by
  \[
   {\textstyle L_{\sigma}(z_1,\ldots,z_p;x):=\sum_{i=1}^p\theta_i(z_i)+\langle x, \sum_{i=1}^p\mathcal{A}_i^*z_i-c \rangle}
    +\frac{\sigma}{2}\big\|{\textstyle\sum_{i=1}^p\mathcal{A}_i^*z_i-c}\big\|^2
  \]
  where $\sigma\!>0$ is the penalty parameter. With a chosen initial point
  $(z_1^0,\ldots,z_p^0;x^0)\in {\rm dom}\,\theta_1\times\cdots\times{\rm dom}\,\theta_p\times\mathbb{X}$,
  the multi-block ADMM consists of the iteration steps:
  \begin{subnumcases}{}
  z_1^{k+1}\in\mathop{\arg\min}_{z_1\in\mathbb{Z}_1}\ L_{\sigma}(z_1,z_2^{k},\ldots,z_p^{k};x^{k}),\nonumber\\
  \qquad\qquad\vdots  \nonumber\\
  \label{ADMM-step1}
  z_i^{k+1}\in\mathop{\arg\min}_{z_i\in\mathbb{Z}_i}\ L_{\sigma}(z_1^{k+1},\ldots,z_{i-1}^{k+1},z_i,z_{i+1}^k,\ldots,z_p^{k};x^{k}),\\
  \qquad\qquad\vdots\nonumber\\
  z_p^{k+1}\in\mathop{\arg\min}_{z_p\in\mathbb{Z}_p}\ L_{\sigma}(z_1^{k+1},\ldots,z_{p-1}^{k+1},z_p;x^{k}),\nonumber\\
  x^{k+1} ={\textstyle x^{k}+\tau\sigma\big(\sum_{i=1}^p\mathcal{A}_i^*z_i^{k+1}- c\big)}.
  \label{ADMM-step2}
  \end{subnumcases}
  Many numerical results have illustrated that the directly extended ADMM with $\tau>1$ works
  very well in many cases (see, e.g., \cite{WGY10,DY2012,HTY12,WHML13,STY14}). In particular,
  Wen et al. \cite{WGY10} have utilized the $3$-block ADMM with $\tau=1.618$ to develop an efficient
  software for solving some SDP problems of large sizes. However, it was shown very recently by
  Chen et al. \cite{CHYY14} that in contrast to the $2$-block ADMM, the directly extended
  $3$-block ADMM may diverge even if $\tau$ is sufficiently small. This dashes any hope of
  using the directly extended multi-block ADMM without any modifications or any restrictions
  on $\theta_i$ or $\mathcal{A}_i$.

  \medskip

  In fact, before the announcement of \cite{CHYY14}, some researchers have made serious attempts
  in correcting the possible divergent multi-block ADMM (see, e.g., \cite{HYZC13,HTXY13,HTY12,HTY14,HY13,DLPY2014}).
  Among them, the multi-block ADMM with Gaussian back substitution \cite{HTY12} distinguishes itself
  for simplicity and generality. However, the recent numerical results reported in \cite{STY14} indicate
  that the multi-block ADMM with Gaussian back substitution (ADMMG for short) requires more iterations
  and computing time than the directly extended ADMM with $\tau=1.618$ for at least $75\%$ test problems,
  and for $61.5\%$ test problems it requires at least $1.5$ times as many iterations as the latter does.
  Now the dilemma is that almost all modified versions of the directly extended ADMM, although with
  convergent guarantee, often perform substantially worse than the directly extended ADMM
  with no convergent guarantee. This paper will make an active attempt in getting out of the dilemma.

  \medskip

  We observe that the ADMMG \cite{HTY12} in each iteration makes a correction on the iterate point
  yielded by the directly extended ADMM with the unit step-size to achieve the global convergence.
  As recognized by the authors in \cite{HTY12}, the introduction of the correction step often
  destroys the good numerical performance of the directly extended ADMM. It is well known that
  the ADMM with $\tau=1$ always requires more $20\%$ to $50\%$ iterations than the one with
  $\tau=1.618$. Thus, there is a big possibility that the iterate points yielded by the directly
  extended ADMM with $\tau=1$ are insufficient to overcome the negative influence of the correction step,
  which may interpret why the ADMMG even with little correction (i.e., the correction step-size
  $\alpha$ is as close to $1$ as possible) usually has worse performance than the directly extended
  ADMM with the unit step-size. Motivated by the crucial observation, we propose a corrected ADMM for
  problem \eqref{prob} by imposing suitable correction only on the middle $(p\!-\!2)$ blocks
  of the iterate point yielded by the multi-block semi-proximal ADMM with a large step-size,
  which is adaptively determined by the infeasibility ratio made up by the current semi-proximal
  ADMM step for the one yielded by the last correction step. Here, the multi-block semi-proximal ADMM,
  instead of the directly extended ADMM, is used to yield the prediction step just for the consideration
  that some subproblems involved in the directly extended ADMM are hard to solve
  but the proximal operators of the corresponding $\theta_i$'s are easy to obtain.

  \medskip

  In contrast to the ADMMG \cite{HTY12} and the linearized ADMMG \cite{HY13},
  our corrected semi-proximal ADMM do not make any correction for the $p$th block and
  the multiplier block of the prediction point. Although the correction step in \cite{HTY12,HY13}
  would not make any correction for the two blocks if the correction step-size takes $1$,
  the global convergence analysis there is not applicable to this extreme case. In addition,
  when the subproblems involved in the directly extended ADMM are easy to solve, one may set
  the semi-proximal operators to be zero, and then the corrected semi-proximal ADMM is using
  the directly extended ADMM to yield the prediction step. However, for this case,
  the linearized ADMMG \cite{HY13} still uses a linearized version of the directly extended ADMM
  to yield the prediction step except that all $\mathcal{A}_i\mathcal{A}_i^*$ reduce to the identity
  and the proximal parameters are all set to be the smallest one $\sigma\|\mathcal{A}_i\mathcal{A}_i^*\|$.
  For the advantage of a semi-proximal term over a strongly convex proximal term,
  the interested readers may refer to \cite{FPST13,STY14}.

  \medskip

  For the proposed corrected semi-proximal ADMM, we provide the global convergence analysis
  under a mild assumption for the operators $\mathcal{A}_i$'s, and apply it to the dual problems
  of five classes of doubly nonnegative SDP problems without linear inequality constraints and
  a class of doubly nonnegative SDP problems with many linear inequality constraints, which take
  the form of \eqref{prob} with $p=3$ and $p=4$, respectively. Our extensive numerical experiments
  for total {\bf 671} test problems demonstrate that the corrected semi-proximal ADMM is superior
  to the directly extended ADMM with $\tau=1.618$ and the ADMMG \cite{HTY12} and
  the linearized ADMMG \cite{HY13}, and it requires the least number of iterations for about
  $70\%$ test instances within the comparable computing time with that of the directly extended ADMM.
  In particular, for about $40\%$ test problems, the number of iterations of the corrected
  semi-proximal ADMM is at most $70\%$ that of the ADMMG \cite{HTY12,HY13}.

  \medskip

  In the rest of this paper, we say that a linear operator $\mathcal{T}\!:\mathbb{X}\to\mathbb{X}$
  is positive semidefinite (respectively, positive definite) if $\mathcal{T}$ is self-adjoint
  and $\langle u,\mathcal{T}\!u\rangle\ge 0$ for any $u\in\mathbb{X}$ (respectively,
  $\langle u,\mathcal{T}\!u\rangle>0$ for any $u\in\mathbb{X}\backslash\{0\}$),
  and write $\|u\|_\mathcal{T}=\sqrt{\langle u,\mathcal{T}\!u\rangle}$ for any $u\in\mathbb{X}$.

 \section{A corrected semi-proximal ADMM}\label{sec2}

  Choose the positive semidefinite linear operators $\mathcal{T}_i\!:\mathbb{Z}_i\to \mathbb{Z}_i$
  for $i=1,2,\ldots,p$ such that all $\mathcal{T}_i+\mathcal{A}_i\mathcal{A}_i^*$ are positive definite.
  Define the mapping $F\!:\mathbb{Z}_1\times\mathbb{Z}_2\times\cdots\times\mathbb{Z}_p\to\mathbb{X}$ by
  \begin{align}\label{hfun}
   F(z_1,z_2,\ldots,z_p):= \mathcal{A}_1^*z_1+\mathcal{A}_2^*z_2+\cdots+\mathcal{A}_p^*z_p-c.\quad
  \end{align}
 Next we describe the detailed iteration steps of the corrected semi-proximal ADMM.

   \begin{algorithm}\label{CSP-ADMM}({\bf Corrected semi-proximal ADMM})
   \begin{description}
    \item[(S.0)] Let $\sigma>0, \alpha\in(0,1)$ and $\overline{\tau}\in\!(0,1)$ be given.
                Choose a suitable small~$\varepsilon\in(0,{1}/{2})$, \hspace*{0.05cm}
                a starting point $\big(z_1^0,\ldots,z_p^0,x^0\big)\in{\rm dom}\,\theta_1\times
                \cdots\times{\rm dom}\,\theta_p\times\mathbb{X}$, and $\tau_0\in(1,2)$.
                Set \hspace*{0.05cm} $\widetilde{z}_i^0=z_i^0$ for $i=1,2,\ldots,p$.
                For $k=0,1,\ldots$, perform the $k$th iteration as follows.

   \item[(S.1)] (\textbf{Semi-proximal ADMM}) Compute the following minimization problems
                \begin{equation}\label{pred-step2}
                 \left\{\begin{array}{l}
                 z_1^{k+1}={\displaystyle\mathop{\arg\min}_{z_1\in\mathbb{Z}_1}}\
                 L_{\sigma}\big(z_1,\widetilde{z}_2^{k},\ldots,\widetilde{z}_{p}^{k};x^{k}\big)+\frac{\sigma}{2}\|z_1-\widetilde{z}_1^k\|^2_{\mathcal{T}_1},\\
                  \qquad\qquad\vdots\\
                  z_i^{k+1}={\displaystyle\mathop{\arg\min}_{z_i\in\mathbb{Z}_i}}\ L_{\sigma}\big(z_1^{k+1},\ldots,z_{i-1}^{k+1},z_i,
                  \widetilde{z}_{i+1}^k,\ldots,\widetilde{z}_{p}^{k};x^{k}\big)+\frac{\sigma}{2}\|z_i-\widetilde{z}_i^k\|^2_{\mathcal{T}_i},\\
                   \qquad\qquad\vdots\\
                  z_p^{k+1}={\displaystyle\mathop{\arg\min}_{z_p\in\mathbb{Z}_p}}\ L_{\sigma}\big(z_1^{k+1},\ldots,z_{p-1}^{k+1},z_p;x^{k}\big)
                  +\frac{\sigma}{2}\|z_p-\widetilde{z}_p^k\|^2_{\mathcal{T}_p},
                  \end{array}\right.
               \end{equation}
               \hspace*{0.05cm} and then update the Lagrange multiplier by the following formula
                \begin{equation}\label{multiplier-update}
                 x^{k+1}=x^{k}+\tau_{k}\sigma\big(\mathcal{A}_1^*z_1^{k+1}
                                +\mathcal{A}_2^*z_2^{k+1} +\cdots+\mathcal{A}_p^*z_p^{k+1}- c\big),
                \end{equation}
              \hspace*{0.05cm} where
                \begin{equation}\label{tauk-rule}
                  \tau_k:=\left\{\!\begin{array}{cl}\!
                                    \min(1\!+\!\delta_k,\tau_{k-1}) & {\rm if}\ 1\!+\!\delta_{k}>\overline{\tau}\\
                                     \overline{\tau}&{\rm otherwise}
                                     \end{array}\right.\ \ {\rm for}\ k\ge 1
                \end{equation}
              \hspace*{0.05cm}   with
               \begin{equation}\label{tauk}
                \delta_{k}=\!\frac{\|F(z_1^{k+1},\widetilde{z}_2^{k},\ldots,\widetilde{z}_{p}^{k})\|^2\!
                                             -\varepsilon\big(\|F(z_1^{k+1},\ldots,z_p^{k+1})\|^2\!+\!\|\mathcal{A}_p^*(z_p^{k+1}\!-\!z_p^k)\|^2\big)}
                                              {\|F(z_1^{k+1},\ldots,z_p^{k+1})\|^2}.
                \end{equation}

    \item[(S.2)] (\textbf{Correction step}) Set $\widetilde{z}_i^{k+1}\!=z_p^{k+1},\widetilde{z}_1^{k+1}\!= z_1^{k+1}$
                  and $\widetilde{z}_i^{k+1}$ for $i=p\!-\!1,\ldots,2$ as
                 \begin{equation}\label{corr-step2}
                 \widetilde{z}_{i}^{k+1}=\widetilde{z}_i^{k} + \alpha(z_i^{k+1}\!-\widetilde{z}_{i}^{k})
                   -{\textstyle\sum_{j=i+1}^{p}}\big(\mathcal{T}_i\!+\!\mathcal{A}_i\mathcal{A}_i^*\big)^{-1}
                  \mathcal{A}_i\mathcal{A}_{j}^*(\widetilde{z}_j^{k+1}\!-\widetilde{z}_j^{k}).
                 \end{equation}

  \item[(S.3)] Let $k\leftarrow k+1$, and then go to Step (S.1).
  \end{description}
 \end{algorithm}

  Since the positive semidefinite linear operators $\mathcal{T}_i$ for $i=1,2,\ldots,p$ are
  chosen such that all $\mathcal{T}_i+\mathcal{A}_i\mathcal{A}_i^*$ are positive definite,
  each subproblem in (S.1) is strongly convex, which implies that Algorithm \ref{CSP-ADMM}
  is well defined. An immediate choice for such $\mathcal{T}_i$ is
  $\varrho_i\mathcal{I}-\mathcal{A}_i\mathcal{A}_i^*$ with $\varrho_i\ge\|\mathcal{A}_i\mathcal{A}_i^*\|$.
  Notice that (S.1) of Algorithm \ref{CSP-ADMM} is using the multi-block semi-proximal ADMM
  to yield a prediction point, which can effectively deal with the case where the subproblems
  involved in \eqref{ADMM-step1} of the directly extended ADMM do not have closed form
  solutions but the proximal operators of $\theta_i$ are easy to obtain.
  The semi-proximal ADMM is clearly proposed just in the recent paper \cite{STY14},
  though the convergence of two-block semi-proximal ADMM was established in the earlier papers \cite{XW11,FPST13}.
  When all $\mathcal{A}_i\mathcal{A}_i^*$ are positive definite, one may choose all $\mathcal{T}_i$
  to be the zero operator and (S.1) of Algorithm \ref{CSP-ADMM} reduces to the directly extended ADMM
  with adaptive step-size.

 \begin{remark}\label{Remark2.2}
  In contrast to the ADMMG in \cite{HTY12} and the linearized ADMMG in \cite{HY13},
  Algorithm \ref{CSP-ADMM} introduces a step-size $\tau_k$ into the multiplier update \eqref{multiplier-update},
  which is adaptively determined by formula \eqref{tauk-rule}-\eqref{tauk}. The $\delta_{k}$ defined in
  \eqref{tauk} actually characterizes the infeasibility ratio made up by the $k$th semi-proximal
  ADMM step for the one yielded by the $(k\!-\!1)$th correction step. When the constant $\varepsilon$
  is chosen to be sufficiently small, the ratio $\delta_{k}$ is always positive, and
  consequently the step-size $\tau_k$ is at least $1$.

  \medskip

  Observe that the multiplier update in the ADMM is same as that of the augmented Lagrangian function method,
  while the latter is an approximate Newton direction when the penalty parameter is over a certain
  threshold (see \cite{SSZ08}). This implies that the multiplier block is good.
  In addition, the block $z_p^{k+1}$ from the semi-proximal ADMM is good since the negative
  influence of the last correction step on it is tiny after the first $(p\!-\!1)$
  minimization of the semi-proximal ADMM. So, unlike the ADMMG \cite{HTY12} and
  the linearized ADMMG \cite{HY13}, Algorithm \ref{CSP-ADMM} does not impose any correction
  on the $p$th block and the multiplier block of the prediction point. Of course, the ADMMG
  would not make any correction for the $p$th block and the multiplier block of the prediction point
  if the correction step-size takes $1$, but the convergence analysis there is not applicable to the extreme case.
 \end{remark}

 \begin{remark}\label{Remark-3block}
  Now let us take a look at a special case with $p=3$, where all $\mathcal{T}_i=0$,
  $\theta_2$ is a linear function, to say $\theta_2(z_2)=\langle b,z_2\rangle$
  for some $b\in\mathbb{Z}_2$, and the operator $\mathcal{A}_2$ is surjective.
  Then, the correction step with the unit step-size reduces to
  \[
    \widetilde{z}_3^{k+1}\!=\!z_3^{k+1},\
    \widetilde{z}_2^{k+1}\!=z_2^{k+1}-(\mathcal{A}_2\mathcal{A}_2^*)^{-1}\mathcal{A}_2\mathcal{A}_3^*(\widetilde{z}_3^{k+1}\!-z_3^{k})
    \ \ {\rm and}\ \ \widetilde{z}_1^{k+1} = z_1^{k+1}.
  \]
  In this case, the iterate $(z_1^{k+1},z_2^{k+1},z_3^{k+1})$
  of Algorithm \ref{CSP-ADMM} is actually yielded by
  \begin{subnumcases}{}\label{equiv-min-prob1}
   z_2^{k+\frac{1}{2}}={\displaystyle\mathop{\arg\min}_{z_2\in\mathbb{Z}_2}}\ L_{\sigma}(z_1^{k},z_2, z_{3}^{k}, x^{k-1}),\\
   \label{equiv-min-prob2}
   z_1^{k+1}={\displaystyle\mathop{\arg\min}_{z_1\in\mathbb{Z}_1}}\ L_{\sigma}(z_1,z_2^{k+\frac{1}{2}}, z_{3}^{k}, x^{k}),\\
   \label{equiv-min-prob3}
   z_2^{k+1}={\displaystyle\mathop{\arg\min}_{z_2\in\mathbb{Z}_2}}\ L_{\sigma}(z_1^{k+1},z_2, z_{3}^{k}, x^{k}),\\
   \label{equiv-min-prob4}
   z_3^{k+1}={\displaystyle\mathop{\arg\min}_{z_3\in\mathbb{Z}_3}}\ L_{\sigma}(z_1^{k+1},z_2^{k+1}, z_{3}, x^{k}),
  \end{subnumcases}
  since it is easy to verify that $z_2^{k+\frac{1}{2}}=\widetilde{z}_2^k$. If, in addition,
  $x^{k-1}$ in \eqref{equiv-min-prob1} is replaced by $x^k$, then the iterate in
  \eqref{equiv-min-prob1}-\eqref{equiv-min-prob4} is equivalent to that of the following two-block ADMM
  \begin{equation}\label{equiv1-min-prob}
   \left\{\!\begin{array}{l}
   z_1^{k+1}={\displaystyle\mathop{\arg\min}_{z_1\in\mathbb{Z}_1}}\ L_{\sigma}\big(z_1,\phi(z_1,z_{3}^k, x^{k}),x^k\big),\\
   z_3^{k+1}={\displaystyle\mathop{\arg\min}_{z_3\in\mathbb{Z}_3}}\ L_{\sigma}\big(z_1^{k+1},\phi(z_1^{k+1},z_{3}^{k}, x^{k}), z_{3}, x^{k}\big)
   \end{array}\right.
  \end{equation}
  with $\phi(z_1,z_3,x):=-(\mathcal{A}_2\mathcal{A}_2^*)^{-1}(\mathcal{A}_1^*z_1\!+\!\mathcal{A}_3z_3-c)
  -\frac{1}{\sigma}(b_2\!+\!\mathcal{A}_2x)$, and \eqref{equiv1-min-prob} is equivalent to
  \begin{equation}\label{equiv2-min-prob}
   \left\{\!\begin{array}{l}
   (z_1^{k+1},z_2^{k+1})={\displaystyle\mathop{\arg\min}_{z_1\in\mathbb{Z}_1,z_2\in\mathbb{Z}_2}}\
    L_{\sigma}\big(z_1,z_2, z_{3}^{k}, x^{k}\big)+\frac{\sigma}{2}\|z_1-z_1^k\|^2_{\mathcal{T}},\\
   z_3^{k+1}={\displaystyle\mathop{\arg\min}_{z_3\in\mathbb{Z}_3}}\ L_{\sigma}\big(z_1^{k+1},z_2^{k+1}, z_{3}, x^{k}\big)
   \end{array}\right.
  \end{equation}
  with $\mathcal{T}=\mathcal{A}_1\mathcal{A}_2^*(\mathcal{A}_2\mathcal{A}_2^*)^{-1}\mathcal{A}_2\mathcal{A}_1^*$.
  The equivalence between the iterate schemes \eqref{equiv1-min-prob} and \eqref{equiv2-min-prob} is recently
  employed by Sun, Toh and Yang \cite{STY14} and Li, Toh and Sun \cite{LST14} to resolve a special case of (\ref{prob})
  in which $p=3$ and one of $\theta_i$'s is linear or quadratic.
  \end{remark}

   \begin{remark}\label{remark-sorder}
    It is immediate to see that the iterate $(\widetilde{z}_1^k,\ldots,\widetilde{z}_p^k,x^k)$ yielded by
    Algorithm \ref{CSP-ADMM} satisfies $\widetilde{z}_1^k\in {\rm dom}\,\theta_1$
    and $\widetilde{z}_p^k\in {\rm dom}\,\theta_p$. Thus, when the constraints
    $z_1\in{\rm dom}\,\theta_1$ and $z_p\in{\rm dom}\,\theta_p$ are hard to be satisfied,
    the best solution order of the subproblems in (S.1) of Algorithm \ref{CSP-ADMM}
    should be as follows: to solve $z_1$ (or $z_p$) first and solve $z_p$ (or $z_1$) last.
  \end{remark}

 \section{Convergence analysis of Algorithm \ref{CSP-ADMM}}

 We need the following constraint qualification where $\Omega$ denotes the feasible set of \eqref{prob}:
 \begin{assumption}\label{assumpA}
  There exists a point $(\widehat{z}_1,\ldots,\widehat{z}_p)\in{\rm ri}({\rm dom}\,\theta_1\times\cdots\times{\rm dom}\,\theta_p)\cap\Omega$.
 \end{assumption}

 Under Assumption \ref{assumpA}, from \cite[Corollary 28.2.2 \& 28.3.1]{Roc70} and \cite[Theorem 6.5 \& 23.8]{Roc70},
 it follows that $(z_1^*,\ldots,z_p^*)\in\mathbb{Z}_1\times\cdots\times\mathbb{Z}_p$
 is an optimal solution to problem \eqref{prob} if and only if there exists a Lagrange multiplier
 $x^*\in\mathbb{X}$ such that
 \begin{equation}\label{optimal-cond}
   -\!\mathcal{A}_ix^*\in\partial\theta_i(z_i^*)\ \ {\rm for}\ i=1,2,\ldots,p\ \ {\rm and}\ \
   \mathcal{A}_1^*z_1^*+\cdots+\mathcal{A}_p^*z_p^*-c=0,
 \end{equation}
 where $\partial\theta_i$ is the subdifferential mapping of $\theta_i$. Moreover, any $x^*\in\mathbb{X}$
 satisfying (\ref{optimal-cond}) is an optimal solution to the dual problem of \eqref{prob}.
 Notice that the subdifferential mapping of a closed proper convex function is maximal monotone by
 \cite[Theorem 12.17]{RW98}. Therefore, for each $i\in\{1,2,\ldots,p\}$, there exists a
 positive semidefinite linear operator $\Sigma_{i}:\mathbb{Z}_i\to\mathbb{Z}_i$ such that
 for all $z_i,\overline{z}_i\in{\rm dom}\,\theta_i$,
 $u_i\in\partial\theta_i(z_i)$ and $\overline{u}_i\in\partial\theta_i(\overline{z}_i)$,
 \begin{equation}\label{subdiff-theta}
   \theta_i(z_i)\ge\theta_i(\overline{z}_i)+\langle \overline{u}_i,z_i-\overline{z}_i\rangle
   +\frac{1}{2}\|z_i-\overline{z}_i\|_{\Sigma_i}^2\ {\rm and}\
   \langle u_i-\overline{u}_i,z_i-\overline{z}_i\rangle\ge \|z_i-\overline{z}_i\|_{\Sigma_i}^2.
 \end{equation}

  First, we establish a technical lemma to deal with the cross terms of the iterates.
  \begin{lemma}\label{lemma1}
  Let $\big\{(z_1^{k},\ldots,z_p^{k},x^{k})\big\}$ and $\big\{(\widetilde{z}_1^{k},\ldots,\widetilde{z}_p^{k})\big\}$
  be the sequences generated by Algorithm \ref{CSP-ADMM}. Then, under Assumption \ref{assumpA},
  for any optimal solution $(z_1^*,\ldots,z_p^*)\in\mathbb{Z}_1\times\cdots\times\mathbb{Z}_p$
  of (\ref{prob}) and the associated Lagrange multiplier $x^*\in\mathbb{X}$, we have
   \begin{align}\label{lemma1-ineq}
    &2\sum_{i=2}^p\!\Big\langle \widetilde{z}_i^{k}-z_i^*,\,{\textstyle\sum_{j=2}^i}\mathcal{A}_i\mathcal{A}_j^*\big(\widetilde{z}_j^k-z_j^{k+1}\big)
      +\mathcal{T}_i(\widetilde{z}_i^k-z_i^{k+1})\Big\rangle \nonumber\\
    &\quad + \frac{2}{\tau_{k}\sigma^2}\big\langle x^{k}-x^{k+1},x^{k}\!-\!x^*\big \rangle
       +2\big\langle z_1^k-z_1^*,\, \mathcal{T}_1(z_1^k-z_1^{k+1})\big\rangle\nonumber\\
     &\ge\Big\|\sum_{i=2}^p\mathcal{A}_i^*\big(z_i^{k+1}-\widetilde{z}_i^{k}\big)-\frac{1}{\tau_{k}\sigma}(x^{k+1}\!-\!x^{k})\Big\|^2
       +\frac{1}{(\tau_{k}\sigma)^2}\big\|x^{k+1}\!-\!x^{k}\big\|^2 \nonumber\\
     &\quad +\sum_{i=2}^p\big\|z_i^{k+1}\!-\!\widetilde{z}^k\big\|_{\mathcal{A}_i\mathcal{A}_i^*+2\mathcal{T}_i}^2
     +2\big\|z_1^{k+1}\!-\!z_1^k\big\|_{\mathcal{T}_1}^2
           +\frac{2}{\sigma}\sum_{i=1}^p\big\|z_i^{k+1}\!-\!z_i^*\big\|_{\Sigma_i}^2\quad \forall k.
   \end{align}
 \end{lemma}
 \begin{proof}
  From the definition of $z_i^{k+1}$ in equation \eqref{pred-step2}, it follows that for $i=1,2,\ldots,p$,
  \begin{equation}\label{opt-cond}
  -\!\mathcal{A}_i\!\left[x^k+\sigma\big({\textstyle \sum_{j=1}^i\mathcal{A}_j^*z_j^{k+1}}
  +{\textstyle\sum_{j=i+1}^p}\mathcal{A}_j^*\widetilde{z}_j^{k} -c\big)\right]
  -\sigma\mathcal{T}_i\big(z_i^{k+1}\!-\!\widetilde{z}_i^k\big)\in\partial\theta_i(z_i^{k+1}).
 \end{equation}
 Since $-\mathcal{A}_ix^*\in \partial \theta_i(z_i^*)$ for $i=1,2,\ldots,p$,
 from equation \eqref{subdiff-theta} we have that
 \begin{align*}
  &\Big\langle z_i^*-z_i^{k+1},\, \mathcal{A}_i\Big[x^k-x^*+\sigma\big({\textstyle \sum_{j=1}^i\mathcal{A}_j^*z_j^{k+1}
  +\!\sum_{j=i+1}^p\!\mathcal{A}_j^*\widetilde{z}_j^{k}} -c\big)\Big]\Big\rangle \nonumber\\
  &\ \ +\sigma\big\langle z_i^*-z_i^{k+1},\, \mathcal{T}_i(z_i^{k+1}-\widetilde{z}_i^k)\big\rangle
  \ge \big\|z_i^{k+1}-z_i^*\big\|_{\Sigma_i}^2,\quad i=1,2,\ldots,p.
 \end{align*}
 Substituting \eqref{multiplier-update} into the last $p$ inequalities successively yields that
 \begin{align}\label{temp-ineq31}
  &\big\langle \mathcal{A}_i^*(z_i^*-z_i^{k+1}),\, x^{k+1}-x^*\big\rangle
  + \sigma\big\langle z_i^*-z_i^{k+1},\, \mathcal{T}_i(z_i^{k+1}-\widetilde{z}_i^k)\big\rangle-\big\|z_i^{k+1}-z_i^*\big\|_{\Sigma_i}^2\nonumber\\
  &\ge\sigma\big\langle \mathcal{A}_i^*(z_i^*-z_i^{k+1}),(\tau_{k}\!-\!1)\big({\textstyle\sum_{j=1}^i\mathcal{A}_j^*z_j^{k+1}-c\big)\!
       +\sum_{j=i+1}^p\mathcal{A}_j^*\big(\tau_kz_j^{k+1}-\widetilde{z}_j^k\big)}\big\rangle
 \end{align}
 for $i=1,2,\ldots,p$. Now adding the term $\sigma\big\langle \mathcal{A}_i^*(z_i^*-z_i^{k+1}),\sum_{j=2}^p\mathcal{A}_j^*\big(\widetilde{z}_j^k-\tau_kz_j^{k+1}\big)\big\rangle$
 to the both sides of the $i$th inequality in \eqref{temp-ineq31} for $i=1,2,\ldots,p$ yields that
 \begin{align}\label{temp-equa31}
  &\big\langle \mathcal{A}_i^*(z_i^*-z_i^{k+1}),\,x^{k+1}-x^*+\sigma{\textstyle\sum_{j=2}^p\mathcal{A}_j^*\big(\widetilde{z}_j^k-\tau_kz_j^{k+1}\big)}\big\rangle
   \nonumber\\
  &\ge\sigma\big\langle \mathcal{A}_i^*(z_i^*-z_i^{k+1}),\,(\tau_{k}\!-\!1)\big(\mathcal{A}_1^*z_1^{k+1}-c\big)
  +{\textstyle\sum_{j=2}^i\mathcal{A}_j^*\big(\widetilde{z}_j^k-z_j^{k+1}\big)}\!\big\rangle\nonumber\\
   &\quad +\sigma\big\langle z_i^*-z_i^{k+1},\, \mathcal{T}_i(\widetilde{z}_i^k-z_i^{k+1})\big\rangle
   +\big\|z_i^{k+1}-z_i^*\big\|_{\Sigma_i}^2,\quad i=1,2,\ldots,p.
  \end{align}
 Adding the $p$ inequalities in \eqref{temp-equa31} together, we have that the left hand side is equal to
 \begin{align*}
   &\big\langle{\textstyle\sum_{i=1}^p\mathcal{A}_i^*(z_i^*-z_i^{k+1})},\,x^{k+1}-x^*
  +\sigma{\textstyle\sum_{j=2}^p\mathcal{A}_j^*\big(\widetilde{z}_j^k-\tau_kz_j^{k+1}\big)}\big\rangle\nonumber\\
  &=\big\langle c-{\textstyle\sum_{i=1}^p\mathcal{A}_i^*z_i^{k+1}},\ x^{k+1}-x^*
     +\sigma{\textstyle\sum_{j=2}^p\mathcal{A}_j^*\big(\widetilde{z}_j^k-\tau_kz_j^{k+1}\big)}\big\rangle,
 \end{align*}
  where the equality is due to $\sum_{i=1}^p\mathcal{A}_i^*z_i^*=c$;
  while the right hand side is equal to
 \begin{align*}
 &\Big\langle c-\!\sum_{i=1}^p\mathcal{A}_i^*z_i^{k+1},\,\sigma(\tau_{k}\!-\!1)(\mathcal{A}_1^*z_1^{k+1}\!-\!c)\Big\rangle
  +\sigma\!\sum_{i=2}^p\!\Big\langle z_i^*-\!z_i^{k+1},\,{\textstyle\sum_{j=2}^i}\mathcal{A}_i^*\mathcal{A}_j^*\big(\widetilde{z}_j^k-z_j^{k+1}\big)\Big\rangle\nonumber\\
   &\quad +\sigma\sum_{i=1}^p\big\langle z_i^*-z_i^{k+1},\, \mathcal{T}_i(\widetilde{z}_i^k-z_i^{k+1})\big\rangle
          +\sum_{i=1}^p\big\|z_i^{k+1}-z_i^*\big\|_{\Sigma_i}^2.
 \end{align*}
 By combining the last two equations with inequality \eqref{temp-equa31}, it then follow that
 \begin{align*}
  &\Big\langle c-\sum_{i=1}^p\mathcal{A}_i^*z_i^{k+1},\,x^{k+1}-x^*\!-\!\sigma(\tau_{k}\!-\!1)(\mathcal{A}_1^*z_1^{k+1}\!-\!c)
   \!+\!\sigma\sum_{j=2}^p\mathcal{A}_j^*\big(\widetilde{z}_j^{k}-\tau_{k}z_j^{k+1}\big)\Big\rangle\nonumber\\
  &\ge\sigma\!\sum_{i=2}^p\!\Big\langle z_i^*-z_i^{k+1},\,{\textstyle\sum_{j=2}^i}\mathcal{A}_i^*\mathcal{A}_j^*\big(\widetilde{z}_j^k-z_j^{k+1}\big)\Big\rangle\nonumber\\
  &\quad\ +\sigma\sum_{i=1}^p\big\langle z_i^*-z_i^{k+1},\, \mathcal{T}_i(\widetilde{z}_i^k-z_i^{k+1})\big\rangle
     +\sum_{i=1}^p\big\|z_i^{k+1}-z_i^*\big\|_{\Sigma_i}^2,
 \end{align*}
 which, by noting that
 \(
  c-\sum_{i=1}^p\mathcal{A}_i^*z_i^{k+1}=\frac{1}{\tau_{k}\sigma}(x^k-x^{k+1}),
 \)
 can be equivalently written as
 \begin{align}\label{temp-ineq1}
 &\frac{1}{\tau_{k}\sigma}\big\langle x^{k}-x^{k+1},x^{k+1}\!-\!x^*\big \rangle+
  \sigma\!\sum_{i=2}^p\!\big\langle z_i^{k+1}-z_i^*,\,{\textstyle\sum_{j=2}^i}\mathcal{A}_i^*\mathcal{A}_j^*\big(\widetilde{z}_j^k-z_j^{k+1}\big)\big\rangle\nonumber\\
 &\ge \frac{1}{\tau_{k}}\Big\langle x^{k+1}-x^k,\,\sum_{i=2}^p\mathcal{A}_i^*\big(\widetilde{z}_i^{k}-\tau_{k}z_i^{k+1}\big)\Big\rangle
      +\frac{\tau_{k}-1}{\tau_{k}}\big\langle x^{k}-x^{k+1},\,\mathcal{A}_1^*z_1^{k+1}\!-\!c\big\rangle\nonumber\\
 &\quad +\sigma\sum_{i=1}^p\big\langle z_i^*-z_i^{k+1},\, \mathcal{T}_i(\widetilde{z}_i^k-z_i^{k+1})\big\rangle
     +\sum_{i=1}^p\big\|z_i^{k+1}-z_i^*\big\|_{\Sigma_i}^2.
 \end{align}
 Next we make a simplification for \eqref{temp-ineq1}. The left hand side of \eqref{temp-ineq1} can be rewritten as
 \begin{align}\label{temp-equa33}
  &\sigma\!\sum_{i=2}^p\!\big\langle \widetilde{z}_i^{k}-z_i^*,\,{\textstyle\sum_{j=2}^i}\mathcal{A}_i\mathcal{A}_j^*\big(\widetilde{z}_j^k-z_j^{k+1}\big)\big\rangle
  + \frac{1}{\tau_{k}\sigma}\big\langle x^{k}-x^{k+1},x^{k}\!-\!x^*\big \rangle \nonumber\\
  &\ +\sigma\!\sum_{i=2}^p\!\big\langle z_i^{k+1}-\widetilde{z}_i^{k},\,{\textstyle\sum_{j=2}^i}\mathcal{A}_i\mathcal{A}_j^*\big(\widetilde{z}_j^k-z_j^{k+1}\big)\big\rangle
  -\frac{1}{\tau_{k}\sigma}\big\|x^{k+1}-x^{k}\big\|^2,
  \end{align}
 while the right hand side of \eqref{temp-ineq1} can be rearranged as follows
 \begin{align*}
 &\frac{1}{\tau_k}\left[\sum_{i=2}^p\big\langle x^{k+1}\!-\!x^{k},\,\mathcal{A}_i^*\big(\widetilde{z}_i^{k}\!-\!z_i^{k+1}\big)\big\rangle
    +(\tau_{k}\!-\!1)\big\langle x^{k}-x^{k+1},\sum_{i=1}^p\mathcal{A}_i^*z_i^{k+1}-c\big\rangle\right]\nonumber\\
 &\quad +\sigma\sum_{i=1}^p\big\langle z_i^*-z_i^{k+1},\, \mathcal{T}_i(\widetilde{z}_i^k-z_i^{k+1})\big\rangle
   +\sum_{i=1}^p\big\|z_i^{k+1}-z_i^*\big\|_{\Sigma_i}^2,\nonumber
 \end{align*}
 which, by using the equality
 \(
   \sum_{i=1}^p\mathcal{A}_i^*z_i^{k+1}\!-c=\frac{1}{\tau_{k}\sigma}(x^{k+1}-x^{k}),
  \)
  is equivalent to
 \begin{align}\label{temp-equa32}
 &\frac{1}{\tau_k}\sum_{i=2}^p\big\langle x^{k+1}\!-\!x^{k},\,\mathcal{A}_i^*\big(\widetilde{z}_i^{k}\!-\!z_i^{k+1}\big)\big\rangle
    -\frac{\tau_{k}-1}{\tau_{k}^2\sigma}\big\|x^{k+1}\!-\!x^{k}\big\|^2+\sigma\sum_{i=1}^p\big\|z_i^{k+1}\!-\!\widetilde{z}_i^k\big\|_{\mathcal{T}_i}^2\nonumber\\
 &\ \  +\sigma\sum_{i=1}^p\big\langle z_i^*\!-\!\widetilde{z}_i^k,\, \mathcal{T}_i(\widetilde{z}_i^k-z_i^{k+1})\big\rangle
       +\sum_{i=1}^p\big\|z_i^{k+1}\!-z_i^*\big\|_{\Sigma_i}^2.
\end{align}
 Now, combining equations \eqref{temp-equa33}-\eqref{temp-equa32} with inequality \eqref{temp-ineq1},
 we obtain that
 \begin{align*}
 &\sum_{i=2}^p\!\big\langle \widetilde{z}_i^{k}-z_i^*,\,{\textstyle\sum_{j=2}^i}\mathcal{A}_i\mathcal{A}_j^*\big(\widetilde{z}_j^k\!-\!z_j^{k+1}\big)\big\rangle
  + \frac{1}{\tau_{k}\sigma^2}\big\langle x^{k}-x^{k+1},x^{k}\!-\!x^*\big \rangle \nonumber\\
 &\ge \sum_{i=2}^p\!\Big\langle \widetilde{z}_i^k-z_i^{k+1},
   \,{\textstyle\sum_{j=2}^i}\mathcal{A}_i\mathcal{A}_j^*\big(\widetilde{z}_j^k-z_j^{k+1}\big)\Big\rangle
  +\frac{1}{(\tau_{k}\sigma)^2}\big\|x^{k+1}\!-\!x^{k}\big\|^2\\
 &\quad +\frac{1}{\tau_k\sigma}\sum_{i=2}^p\big\langle x^{k+1}-x^{k},\,\mathcal{A}_i^*\big(\widetilde{z}_i^{k}-z_i^{k+1}\big)\big\rangle
 +\frac{1}{\sigma}\sum_{i=1}^p\big\|z_i^{k+1}-z_i^*\big\|_{\Sigma_i}^2\nonumber\\
  &\quad +\sum_{i=1}^p\big\|z_i^{k+1}-\widetilde{z}_i^k\big\|_{\mathcal{T}_i}^2
       +\sum_{i=1}^p\big\langle z_i^*\!-\!\widetilde{z}_i^k,\, \mathcal{T}_i(\widetilde{z}_i^k-z_i^{k+1})\big\rangle\\
 &=\frac{1}{2}\sum_{i=2}^p\big\|\mathcal{A}_i^*({z}_i^{k+1}-\widetilde{z}_i^k)\big\|^2
   +\frac{1}{2}\Big\|\sum_{i=2}^p\mathcal{A}_i^*\big(z_i^{k+1}-\widetilde{z}_i^{k}\big)-\frac{1}{\tau_{k}\sigma}(x^{k+1}\!-\!x^{k})\Big\|^2\\
 &\quad +\frac{1}{2(\tau_{k}\sigma)^2}\big\|x^{k+1}\!-\!x^{k}\big\|^2+\frac{1}{\sigma}\sum_{i=1}^p\big\|z_i^{k+1}\!-z_i^*\big\|_{\Sigma_i}^2\\
 &\quad +\sum_{i=1}^p\big\|z_i^{k+1}-\widetilde{z}_i^k\big\|_{\mathcal{T}_i}^2
       +\sum_{i=1}^p\big\langle z_i^*\!-\!\widetilde{z}_i^k,\, \mathcal{T}_i(\widetilde{z}_i^k-z_i^{k+1})\big\rangle.
 \end{align*}
  This along with $\widetilde{z}_1^k=z_1^k$ implies the desired inequality. The proof is completed.
 \end{proof}

 \medskip

 To establish the convergence results of Algorithm \ref{CSP-ADMM}, we introduce the notations
 \[
    w^*=(z_2^*,\ldots,z_p^*),\ \
    \mathcal{E}_i=\mathcal{A}_i\mathcal{A}_i^*+\mathcal{T}_i\ \ {\rm and}\ \
    \mathcal{B}_i=\mathcal{A}_i\mathcal{E}_i^{-1}\mathcal{A}_i^*\ \ {\rm for}\ \ i=2,3,\ldots,p.
 \]
 For each $k$, let $w^k=(z_2^k,z_3^k,\ldots,z_p^k)$ and $\widetilde{w}^k=(\widetilde{z}_2^k,\widetilde{z}_3^k,\ldots,\widetilde{z}_p^k)$.
 Define the linear operators $\mathcal{M}\!:\mathbb{Z}_2\times\cdots\times\mathbb{Z}_p\to
  \mathbb{Z}_2\times\cdots\times\mathbb{Z}_p$ and $\mathcal{H}\!:\mathbb{Z}_2\times\cdots\times\mathbb{Z}_p\to
  \mathbb{Z}_2\times\cdots\times\mathbb{Z}_p$, respectively, by
  \[
     \mathcal{M}:=\!\left[\begin{matrix}
                  \mathcal{E}_2& 0 &\cdots & 0 & 0\\
                  \mathcal{A}_3\mathcal{A}_2^* & \mathcal{E}_3 &\cdots & 0 & 0\\
                  \vdots & \vdots & \vdots & \vdots & \vdots\\
                  \mathcal{A}_{p-1}\mathcal{A}_2^* & \mathcal{A}_{p-1}\mathcal{A}_3^* &\cdots&\mathcal{E}_{p-1}& 0\\
                  \mathcal{A}_p\mathcal{A}_2^* & \mathcal{A}_p\mathcal{A}_3^* &\cdots&\mathcal{A}_p\mathcal{A}_{p-1}^*& \mathcal{E}_p \\
                  \end{matrix}\right]
  \]
  and
  \[
  \mathcal{H}:=\left[\begin{matrix}
   \mathcal{I}& \mathcal{E}_2^{-1}\mathcal{A}_2\mathcal{A}_3^*
   & \cdots &\mathcal{E}_2^{-1}\mathcal{A}_2\mathcal{A}_{p-1}^*& \mathcal{E}_2^{-1}\mathcal{A}_2\mathcal{A}_p^* \\
    0 & \mathcal{I} & \cdots&\mathcal{E}_3^{-1}\mathcal{A}_3\mathcal{A}_{p-1}^*& \mathcal{E}_3^{-1}\mathcal{A}_3\mathcal{A}_p^*\\
    \vdots & \vdots & \vdots &\vdots& \vdots \\
    0 & 0 & \cdots&\mathcal{I}&\mathcal{E}_{p-1}^{-1}\mathcal{A}_{p-1}^*\mathcal{A}_p\\
    0 & 0 & \cdots & 0& \alpha\mathcal{I}\\
   \end{matrix}\right].
 \]
 An elementary computation yields that the operator $\mathcal{G}:=\mathcal{M}\mathcal{H}$ takes the form of
  \begin{equation*}
   \left[\begin{matrix}
   \mathcal{E}_2& \mathcal{A}_2\mathcal{A}_3^*
   & \cdots &\mathcal{A}_2\mathcal{A}_{p-1}^*& \mathcal{A}_2\mathcal{A}_p^*\\
   \mathcal{A}_3\mathcal{A}_2^* & \mathcal{A}_3(\mathcal{I}\!+\!\mathcal{B}_2)\mathcal{A}_3^*\!+\!\mathcal{T}_3 &
   \cdots&\mathcal{A}_3(\mathcal{I}\!+\!\mathcal{B}_2)\mathcal{A}_{p-1}^*& \mathcal{A}_3(\mathcal{I}\!+\!\mathcal{B}_2)\mathcal{A}_p^*\\
    \vdots & \vdots & \vdots &\vdots& \vdots \\
    \mathcal{A}_p\mathcal{A}_2^* & \mathcal{A}_p(\mathcal{I}\!+\!\mathcal{B}_2)\mathcal{A}_3^* & \cdots
    & \mathcal{A}_{p}(\mathcal{I}\!+\!\sum_{j=2}^{p-2}\mathcal{B}_j)\mathcal{A}_{p-1}
    & \mathcal{A}_p(\alpha\mathcal{I}\!+\!\sum_{j=2}^{p-1}\mathcal{B}_j)\mathcal{A}_p^*\!+\!\alpha\mathcal{T}_p\\
   \end{matrix}\right].
  \end{equation*}
  It is not hard to verify that the self-adjoint linear operator $\mathcal{G}$ is positive definite.

  \medskip

  Now we are in a position to establish the global convergence of Algorithm \ref{CSP-ADMM}.
 \begin{theorem}\label{convergence1}
  Suppose that Assumption \ref{assumpA} holds and the operators $\mathcal{T}_i$ for $i=1,2,\ldots,p$
  are chosen such that $\mathcal{A}_i\mathcal{A}_i^*+\mathcal{T}_i$ are positive definite.
  Then, the following statements hold:
 \begin{itemize}
  \item[(a)] $\lim\limits_{k \rightarrow +\infty}\big\|z_i^{k+1}- \widetilde{z}_i^k\big\| = 0$ for $i=2,3,\ldots,p$
              and $\lim\limits_{k \rightarrow +\infty}\big\|x^{k+1}-x^k\big\|=0$.

  \item[(b)] The sequences $\big\{(z_1^{k},\ldots,z_p^{k})\big\}$ and
             $\big\{(\widetilde{z}_1^{k},\ldots,\widetilde{z}_p^{k})\big\}$ converge to an optimal solution to (\ref{prob}),
             and $\{x^{k}\}$ converges to an optimal solution to the dual problem of (\ref{prob}).
  \end{itemize}
  \end{theorem}
 \begin{proof}
  Let $(z_1^*,\ldots,z_p^*)\in\mathbb{Z}_1\times\cdots\times\mathbb{Z}_p$ be an optimal
  solution to (\ref{prob}) and $x^*\in\mathbb{X}$ be the associated Lagrange multiplier.
  Then, the sequences $\big\{(z_1^{k},\ldots,z_p^{k},x^{k})\big\}$ and
  $\big\{(\widetilde{z}_1^{k},,\ldots,\widetilde{z}_p^{k})\big\}$ generated by
  Algorithm \ref{CSP-ADMM} satisfies the inequality \eqref{lemma1-ineq} of Lemma \ref{lemma1}.
  By using the expression of the above linear operator $\mathcal{M}$, it is not hard to obtain that
 \begin{align}\label{theo-equa31}
    &2\sum_{i=2}^p\!\Big\langle \widetilde{z}_i^{k}-z_i^*,\,{\textstyle\sum_{j=2}^i}\mathcal{A}_i\mathcal{A}_j^*\big(\widetilde{z}_j^k-z_j^{k+1}\big)
     +\mathcal{T}_i(\widetilde{z}_i^k-z_i^{k+1})\Big\rangle\nonumber\\
   &=2\big\langle \widetilde{w}^{k}\!-\!w^*,\mathcal{M}\big(\widetilde{w}^{k}\!-\!w^{k+1}\big)\big\rangle
    =2\big\langle \widetilde{w}^{k}\!-\!w^*,\mathcal{M}\mathcal{H}\mathcal{H}^{-1}\big(\widetilde{w}^{k}\!-\!w^{k+1}\big)\big\rangle.
  \end{align}
  From the expression of $\mathcal{H}$ and the corrected step of Algorithm \ref{CSP-ADMM},
  we can verify that
  \begin{equation}\label{temp-equa34}
    \mathcal{H}(\widetilde{w}^{k+1}-\widetilde{w}^k) = \alpha(w^{k+1}-\widetilde{w}^k),
  \end{equation}
  which by the invertibility of $\mathcal{H}$ implies that
  $\widetilde{w}^k-\widetilde{w}^{k+1}=\alpha\mathcal{H}^{-1}(\widetilde{w}^k-w^{k+1})$.
  Hence,
  \begin{align}\label{theo-equa32}
   &2\big\langle \widetilde{w}^{k}\!-\!w^*,\mathcal{M}\mathcal{H}\mathcal{H}^{-1}(\widetilde{w}^{k}-\!w^{k+1}\big)\big\rangle
   =2\alpha^{-1}\big\langle \widetilde{w}^{k}\!-\!w^*,\mathcal{G}(\widetilde{w}^k-\widetilde{w}^{k+1})\big\rangle\nonumber\\
   &=\alpha^{-1}\big\|\widetilde{w}^{k}\!-\!w^*\big\|_{\mathcal{G}}^2
     -\alpha^{-1}\big\|\widetilde{w}^{k+1}\!-\!w^*\big\|_{\mathcal{G}}^2
     \!+\alpha\big\|\mathcal{H}^{-1}(\widetilde{w}^{k}\!-\!w^{k+1}\big)\big\|_{\mathcal{G}}^2
  \end{align}
  where the second equality is using the following identity relation
  \begin{equation}\label{identity}
   2\langle u-v,\mathcal{T}(u-w)\rangle=\big\|u-v\big\|_{\mathcal{T}}^2+\big\|u-w\big\|_{\mathcal{T}}^2-\big\|v-w\big\|_{\mathcal{T}}^2.
  \end{equation}
  for a positive semidefinite linear operator $\mathcal{T}$. Using the identity \eqref{identity}, we also have
  \begin{align}\label{theo-equa33}
  &\frac{2}{\tau_{k}\sigma^2}\big\langle x^{k}-x^{k+1},x^{k}\!-\!x^*\big \rangle
      +2\big\langle z_1^k-z_1^*,\, \mathcal{T}_1(z_1^k-z_1^{k+1})\big\rangle\nonumber\\
  &=\frac{1}{\tau_{k}\sigma^2}\left(\big\|x^{k+1}-x^k\big\|^2+\big\|x^k-x^*\big\|^2-\big\|x^{k+1}-x^*\big\|^2\right)\nonumber\\
  &\quad\ +\big\|z_1^{k}-z_1^*\big\|_{\mathcal{T}_1}^2-\big\|z_1^{k+1}-z_1^*\big\|_{\mathcal{T}_1}^2
  +\big\|z_1^{k+1}-z_1^k\big\|_{\mathcal{T}_1}^2.
  \end{align}
  By combining equations \eqref{theo-equa31}, \eqref{theo-equa32} and \eqref{theo-equa33} with inequality \eqref{lemma1-ineq},
  it follows that
  \begin{align}\label{theo-equa3}
   &\frac{1}{\alpha}\left(\big\|\widetilde{w}^{k}\!-\!w^*\big\|_{\mathcal{G}}^2 -\big\|\widetilde{w}^{k+1}\!-\!w^*\big\|_{\mathcal{G}}^2\right)
   +\alpha\big\|\mathcal{H}^{-1}(\widetilde{w}^{k}\!-\!w^{k+1}\big)\big\|_{\mathcal{G}}^2+\big\|z_1^{k}\!-\!z_1^*\big\|_{\mathcal{T}_{1}}^2\nonumber \\
   &\quad -\big\|z_1^{k+1}-z_1^*\big\|_{\mathcal{T}_{1}}^2
   +\frac{1}{\tau_{k}\sigma^2}\left(\big\|x^{k+1}-x^k\big\|^2+\big\|x^k-x^*\big\|^2-\big\|x^{k+1}-x^*\big\|^2\right)\nonumber\\
   &\ge \Big\|\sum_{i=2}^p\mathcal{A}_i^*(z_i^{k+1}\!-\widetilde{z}_i^{k})-\frac{1}{\tau_{k}\sigma}(x^{k+1}-x^k)\Big\|^2
        +\frac{1}{(\tau_{k}\sigma)^2}\big\|x^{k+1}\!-\!x^{k}\big\|^2\nonumber\\
   &\quad +\sum_{i=2}^p\big\|z_i^{k+1}-\widetilde{z}_i^k\big\|_{\mathcal{A}_i\mathcal{A}_i^*+2\mathcal{T}_i}^2
          +\big\|z_1^{k+1}-z_1^k\big\|_{\mathcal{T}_1}^2
        +\frac{2}{\sigma}\sum_{i=1}^p\big\|z_i^{k+1}\!-\!z_i^*\big\|_{\Sigma_i}^2.
 \end{align}
  By the expressions of the operators $\mathcal{H}$ and $\mathcal{M}$, an elementary computation yields that
   \begin{equation*}
   \mathcal{M}^*\mathcal{H}^{-1}=
   \left[\begin{matrix}
    \mathcal{A}_2\mathcal{A}_2^*& 0 & \cdots &0& 0\\
   0 & \mathcal{A}_3\mathcal{A}_3^* & \cdots&0& 0\\
    \vdots & \vdots & \vdots &\vdots& \vdots \\
    0 & 0 & \cdots
    & \mathcal{A}_{p-1}\mathcal{A}_{p-1}^*&0\\
    0 & 0 &\cdots&0 &\frac{1}{\alpha}\mathcal{A}_p\mathcal{A}_p^*\\
   \end{matrix}\right],
  \end{equation*}
  and consequently
  \begin{align*}
   \alpha\big\|\mathcal{H}^{-1}(\widetilde{w}^{k}\!-\!w^{k+1})\big\|_{\mathcal{G}}^2
    &=\alpha\big\langle (\widetilde{w}^{k}\!-\!w^{k+1}),
    (\mathcal{H}^{-1})^*\mathcal{G}\mathcal{H}^{-1}(\widetilde{w}^{k}\!-\!w^{k+1})\big\rangle\nonumber\\
   &=\alpha\big\langle \mathcal{M}^*\mathcal{H}^{-1}(\widetilde{w}^{k}\!-\!w^{k+1}),
    (\widetilde{w}^{k}\!-\!w^{k+1})\big\rangle\nonumber\\
   &=\alpha{\textstyle\sum_{i=2}^{p-1}}\big\|\mathcal{A}_i^*(z_i^{k+1}-\widetilde{z}_i^{k})\big\|^2
     +\big\|\mathcal{A}_p^*(z_p^{k+1}-z_p^{k})\big\|^2.\nonumber
  \end{align*}
  Substituting this equality into inequality (\ref{theo-equa3}), we immediately obtain that
  \begin{align*}
  &\frac{1}{\alpha}\left(\big\|\widetilde{w}^{k}\!-\!w^*\big\|_{\mathcal{G}}^2 -\big\|\widetilde{w}^{k+1}\!-\!w^*\big\|_{\mathcal{G}}^2\right)
  +\big\|z_1^{k}-z_1^*\big\|_{\mathcal{T}_{1}}^2\nonumber\\
  &\quad -\big\|z_1^{k+1}-z_1^*\big\|_{\mathcal{T}_{1}}^2+\frac{1}{\tau_k\sigma^2}\left(\big\|x^k-x^*\big\|^2-\big\|x^{k+1}-x^*\big\|^2\right)\nonumber\\
  &\ge\sum_{i=2}^{p-1}\big\|z_i^{k+1}\!-\!\widetilde{z}_i^k\big\|_{(1-\alpha)\mathcal{A}_i\mathcal{A}_i^*+2\mathcal{T}_{i}}^2
     +\big\|z_1^{k+1}\!-\!z_1^k\big\|_{\mathcal{T}_1}^2+2\big\|z_p^{k+1}\!-\!z_p^k\big\|_{\mathcal{T}_p}^2\nonumber\\
  &\quad +\Big\|\sum_{i=2}^p\mathcal{A}_i^*(z_i^{k+1}\!-\!\widetilde{z}_i^{k})-\frac{x^{k+1}\!-\!x^k}{\tau_{k}\sigma}\Big\|^2
        +\frac{1\!-\!\tau_k}{(\tau_{k}\sigma)^2}\big\|x^{k+1}\!-\!x^{k}\big\|^2
        +\frac{2}{\sigma}\sum_{i=1}^{p}\big\|z_i^{k+1}\!-\!z_i^*\big\|_{\Sigma_i}^2.
  \end{align*}
  Let $\widetilde{\mathcal{G}}$ be an operator with the same form as $\mathcal{G}$
  except that the $p$th diagonal element is replaced by
  $\mathcal{A}_p(\alpha\mathcal{I}\!+\!\sum_{j=2}^{p-1}\mathcal{B}_j)\mathcal{A}_p^*+\alpha\mathcal{T}_p+\frac{2\alpha}{\sigma}\Sigma_p$.
  Then the last inequality is equivalent to
 \begin{align}\label{theo-equa4}
   &\frac{1}{\alpha}\left(\big\|\widetilde{w}^{k}\!-\!w^*\big\|_{\widetilde{\mathcal{G}}}^2
                 -\big\|\widetilde{w}^{k+1}\!-\!w^*\big\|_{\widetilde{\mathcal{G}}}^2\right)
   +\frac{1}{\tau_k\sigma^2}\left(\big\|x^k-x^*\big\|^2-\big\|x^{k+1}-x^*\big\|^2\right)\nonumber\\
  &\quad +\big\|z_1^{k}-z_1^*\big\|_{\mathcal{T}_{1}+\frac{2}{\sigma}\Sigma_1}^2
          -\big\|z_1^{k+1}-z_1^*\big\|_{\mathcal{T}_{1}+\frac{2}{\sigma}\Sigma_1}^2\nonumber\\
 &\ge\sum_{i=2}^{p-1}\big\|z_i^{k+1}-\widetilde{z}_i^k\big\|_{(1-\alpha)\mathcal{A}_i\mathcal{A}_i^*+2\mathcal{T}_{i}}^2
     +\big\|z_1^{k+1}-z_1^k\big\|_{\mathcal{T}_1}^2+2\big\|z_p^{k+1}-z_p^k\big\|_{\mathcal{T}_p}^2\nonumber\\
  &\quad +\Big\|\sum_{i=2}^p\mathcal{A}_i^*(z_i^{k+1}\!-\widetilde{z}_i^{k})-\frac{1}{\tau_{k}\sigma}(x^{k+1}-x^k)\Big\|^2
         +\frac{1-\tau_k}{(\tau_{k}\sigma)^2}\big\|x^{k+1}-x^{k}\big\|^2\nonumber\\
  &\qquad +\frac{2}{\sigma}\big\|z_1^{k}-z_1^*\big\|_{\Sigma_1}^2
         +\frac{2}{\sigma}\big\|z_p^{k}-z_p^*\big\|_{\Sigma_p}^2
         +\frac{2}{\sigma}\sum_{i=2}^{p-1}\big\|z_i^{k+1}\!-\!z_i^*\big\|_{\Sigma_i}^2.
  \end{align}
  Recall that $F(z_1,z_2,\ldots,z_p)=\sum_{i=1}^p\mathcal{A}_i^*z_i-c$. Therefore, we have that
  \[
    \sum_{i=2}^p\mathcal{A}_i^*(z_i^{k+1}\!-\widetilde{z}_i^{k})
    =F(z_1^{k+1},z_2^{k+1},\ldots,z_p^{k+1})-F(z_1^{k+1},\widetilde{z}_2^{k},\ldots,\widetilde{z}_p^{k}).
  \]
  In addition, from equation (\ref{multiplier-update}) it follows that
  \(
    x^{k+1}-x^k=\sigma\tau_{k}F(z_1^{k+1},z_2^{k+1},\ldots,z_p^{k+1}).
  \)
  Substituting the two equalities into (\ref{theo-equa4}), we obtain the desired inequality
  \begin{align}\label{theo-result}
     &\frac{1}{\alpha}\left(\big\|\widetilde{w}^{k}\!-\!w^*\big\|_{\widetilde{\mathcal{G}}}^2
                   -\big\|\widetilde{w}^{k+1}\!-\!w^*\big\|_{\widetilde{\mathcal{G}}}^2\right)
     +\frac{1}{\tau_k\sigma^2}\left(\big\|x^k\!-\!x^*\big\|^2-\big\|x^{k+1}\!-\!x^*\big\|^2\right)\nonumber\\
    &\quad +\big\|z_1^{k}-z_1^*\big\|_{\mathcal{T}_{1}+\frac{2}{\sigma}\Sigma_1}^2
            -\big\|z_1^{k+1}-z_1^*\big\|_{\mathcal{T}_{1}+\frac{2}{\sigma}\Sigma_1}^2\nonumber\\
    &\ge\sum_{i=2}^{p-1}\big\|z_i^{k+1}-\widetilde{z}_i^k\big\|_{(1-\alpha)\mathcal{A}_i\mathcal{A}_i^*+2\mathcal{T}_{i}}^2
       +\big\|z_1^{k+1}-z_1^k\big\|_{\mathcal{T}_1}^2+2\big\|z_p^{k+1}-z_p^k\big\|_{\mathcal{T}_p}^2\nonumber\\
     &\quad +\big\|F(z_1^{k+1},\widetilde{z}_2^{k},\ldots,\widetilde{z}_p^{k})\big\|^2
             + (1-\tau_k)\big\|F(z_1^{k+1},z_2^{k+1},\ldots,z_p^{k+1})\big\|^2\nonumber\\
    &\qquad +\frac{2}{\sigma}\big\|z_1^{k}-z_1^*\big\|_{\Sigma_1}^2
           +\frac{2}{\sigma}\big\|z_p^{k}-z_p^*\big\|_{\Sigma_p}^2
           +\frac{2}{\sigma}\sum_{i=2}^{p-1}\big\|z_i^{k+1}\!-\!z_i^*\big\|_{\Sigma_i}^2.
  \end{align}

  By the definition of $\tau_k$ in \eqref{tauk-rule}, we have that the sequence $\{\tau_k\}$ is
  nonincreasing and $\tau_k=\overline{\tau}$ for all $k\ge \overline{k}$ once $\tau_{\overline{k}}=\overline{\tau}$.
  In view of this, we next prove the results of part (a) and part (b) by the case
  where $\tau_k>\overline{\tau}$ for all $k$ or $\tau_k=\overline{\tau}$ for all $k\ge \overline{k}$.

  \medskip
  \noindent
  {\bf Case 1: $\tau_k>\overline{\tau}$ for all $k$}. In this case, the definition of $\tau_k$ implies
  $\delta_k\ge 1-\tau_k$, and then
  \begin{align*}
    &\big\|F(z_1^{k+1},\widetilde{z}_2^{k},\ldots,\widetilde{z}_p^{k})\big\|^2
     +(1\!-\!\tau_{k})\big\|F(z_1^{k+1},z_2^{k+1},\ldots,z_p^{k+1})\big\|^2\nonumber\\
    &\ge \varepsilon\big\|F(z_1^{k+1},z_2^{k+1},\ldots,z_p^{k+1})\big\|^2
         +\varepsilon\big\|\mathcal{A}_p^*(z_{p}^{k+1}-z_p^{k})\big\|^2.
  \end{align*}
  Together with the above inequality \eqref{theo-result}, we immediately obtain that
  \begin{align}\label{theo-ineq31}
   &\frac{\tau_k}{\alpha}\left(\big\|\widetilde{w}^{k}\!-\!w^*\big\|_{\widetilde{\mathcal{G}}}^2
                 -\big\|\widetilde{w}^{k+1}\!-\!w^*\big\|_{\widetilde{\mathcal{G}}}^2\right)
   +\frac{1}{\sigma^2}\left(\big\|x^k\!-\!x^*\big\|^2-\big\|x^{k+1}\!-\!x^*\big\|^2\right)\nonumber\\
   &\quad +\tau_k\big\|z_1^{k}-z_1^*\big\|_{\mathcal{T}_{1}+\frac{2}{\sigma}\Sigma_1}^2
          -\tau_k\big\|z_1^{k+1}-z_1^*\big\|_{\mathcal{T}_{1}+\frac{2}{\sigma}\Sigma_1}^2\nonumber\\
  &\ge \tau_k\sum_{i=2}^{p-1}\big\|z_i^{k+1}-\widetilde{z}_i^k\big\|_{(1-\alpha)\mathcal{A}_i\mathcal{A}_i^*+2\mathcal{T}_{i}}^2
       +\tau_k\big\|z_1^{k+1}-z_1^k\big\|_{\mathcal{T}_1}^2\nonumber\\
   &\quad +\tau_k\varepsilon\big\|F(z_1^{k+1},z_2^{k+1},\ldots,z_p^{k+1})\big\|^2
         +\tau_k\big\|z_{p}^{k+1}-z_p^{k}\big\|_{2\mathcal{T}_p+\varepsilon\mathcal{A}_p\mathcal{A}_p^*}^2\nonumber\\
   &\quad +\frac{2\tau_k}{\sigma}\big\|z_1^{k}-z_1^*\big\|_{\Sigma_1}^2
          +\frac{2\tau_k}{\sigma}\big\|z_p^{k}-z_p^*\big\|_{\Sigma_p}^2
          +\frac{2\tau_k}{\sigma}\sum_{i=2}^{p-1}\big\|z_i^{k+1}\!-\!z_i^*\big\|_{\Sigma_i}^2.
 \end{align}
  Notice that $\tau_{k+1}\le \tau_{k}$ for all $k\ge 1$ in this case. Therefore, we have that
 \begin{align}\label{theo-ineq32}
  &\sum_{k=0}^\infty\tau_k\left\{\sum_{i=2}^{p-1}\big\|z_i^{k+1}-\widetilde{z}_i^k\big\|_{(1-\alpha)\mathcal{A}_i\mathcal{A}_i^*+2\mathcal{T}_{i}}^2
       +\big\|z_1^{k+1}-z_1^k\big\|_{\mathcal{T}_1}^2+\varepsilon\big\|F(z_1^{k+1},\ldots,z_p^{k+1})\big\|^2\right.\nonumber\\
   &\quad \left. +\big\|z_{p}^{k+1}\!-\!z_p^{k}\big\|_{2\mathcal{T}_p + \varepsilon\mathcal{A}_p\mathcal{A}_p^*}^2
          +\frac{2}{\sigma}\big\|z_1^{k}-z_1^*\big\|_{\Sigma_1}^2
          +\frac{2}{\sigma}\big\|z_p^{k}-z_p^*\big\|_{\Sigma_p}^2
          +\frac{2}{\sigma}\sum_{i=2}^{p-1}\big\|z_i^{k+1}-\!z_i^*\big\|_{\Sigma_i}^2\right\}\nonumber\\
  &\le \frac{1}{\alpha}\sum_{k=0}^\infty\left[\tau_k\big\|\widetilde{w}^{k}\!-\!w^*\big\|^2
               -\tau_{k+1}\big\|\widetilde{w}^{k+1}\!-\!w^*\big\|^2\right]
        +\frac{1}{\sigma^2}\sum_{k=0}^\infty\left[\big\|x^{k}-x^*\big\|^2\!-\!\big\|x^{k+1}-x^*\big\|^2\right]\nonumber\\
  &\qquad +\sum_{k=0}^\infty\left[\tau_k\big\|z_1^{k}-z_1^*\big\|_{\mathcal{T}_{1}+\frac{2}{\sigma}\Sigma_1}^2
     -\tau_{k+1}\big\|z_1^{k+1}-z_1^*\big\|_{\mathcal{T}_{1}+\frac{2}{\sigma}\Sigma_1}^2\right]\nonumber\\
  &\le \frac{\tau_0}{\alpha}\big\|\widetilde{w}^{0}-w^*\big\|^2+\frac{1}{\sigma^2}\big\|x^0-x^*\big\|^2
       +\tau_0\big\|z_1^{0}-z_1^*\big\|_{\mathcal{T}_{1}+\frac{2}{\sigma}\Sigma_1}^2.
 \end{align}
  Since $\alpha\in(0,1)$ and $\tau_k>\overline{\tau}>0$, from inequality \eqref{theo-ineq32} it follows that
  \begin{equation}\label{theo-limit1}
    \lim_{k\to+\infty}\left[{\textstyle
    \sum_{i=2}^{p-1}\big\|z_i^{k+1}-\widetilde{z}_i^k\big\|_{(1-\alpha)\mathcal{A}_i\mathcal{A}_i^*+2\mathcal{T}_{i}}^2}
    +\big\|z_{p}^{k+1}\!-\!z_p^{k}\big\|_{2\mathcal{T}_p+\varepsilon\mathcal{A}_p\mathcal{A}_p^*}^2\right]=0,
  \end{equation}
  which, together with the choice of $\mathcal{T}_i$ for $i=2,3,\ldots,p$, implies that
  \begin{equation}\label{theo-equa34}
    \lim_{k\to+\infty}\|z_i^{k+1}\!-\widetilde{z}_i^{k}\|=0,\quad i=2,3,\ldots,p.
  \end{equation}
  Notice that inequality \eqref{theo-ineq32} also implies that
  \(
    \lim_{k\to\infty}\big\|F(z_1^{k+1},\ldots,z_p^{k+1})\big\|^2=0.
  \)
  This, together with equation \eqref{multiplier-update} and $\tau_k\le\tau_0$, yields that
  \begin{equation}\label{theo-equa35}
   \lim_{k\to+\infty}\big\|x^{k+1}-x^k\big\|=\lim_{k\to+\infty}\big\|F(z_1^{k+1},z_2^{k+1},\ldots,z_p^{k+1})\big\|=0.
  \end{equation}
  The last two equations show that the results of part (a) hold.
  We next prove that the conclusions of part (b) hold.
  Notice that equation \eqref{theo-ineq31} and $\tau_{k+1}\le\tau_k$ imply that
  \begin{align*}
    &\frac{\tau_k}{\alpha}\big\|\widetilde{w}^{k}-w^*\big\|_{\widetilde{\mathcal{G}}}^2 +\frac{1}{\sigma^2}\big\|x^k-x^*\big\|^2
    +\tau_k\big\|z_1^{k}-z_1^*\big\|_{\mathcal{T}_{1}+\frac{2}{\sigma}\Sigma_1}^2\\
    &\ge\frac{\tau_{k+1}}{\alpha}\big\|\widetilde{w}^{k+1}-w^*\big\|_{\widetilde{\mathcal{G}}}^2+\frac{1}{\sigma^2}\big\|x^{k+1}-x^*\big\|^2
    +\tau_{k+1}\big\|z_1^{k+1}-z_1^*\big\|_{\mathcal{T}_{1}+\frac{2}{\sigma}\Sigma_1}^2.
  \end{align*}
  Hence, the sequence $\big\{\frac{\tau_k}{\alpha}\|\widetilde{w}^{k}\!-\!w^*\|_{\widetilde{\mathcal{G}}}^2
  +\frac{1}{\sigma^2}\|x^k\!-\!x^*\|^2 +\tau_k\|z_1^{k}\!-\!z_1^*\|_{\mathcal{T}_{1}+\frac{2}{\sigma}\Sigma_1}^2\!\big\}$
  is convergent, which implies that the sequences $\{\widetilde{w}^k\}$ and $\{x^k\}$ are bounded,
  and $\big\{\|z_1^{k}\!-\!z_1^*\|_{\mathcal{T}_{1}+\frac{2}{\sigma}\Sigma_1}^2\!\big\}$ is bounded.
  Together with $\lim_{k\to+\infty}\|w^{k+1}\!-\widetilde{w}^{k}\|=0$ in part (a), it follows that
  $\{w^k\}$ is bounded. From the boundedness of $\{w^k\}$ and $\{F(z_1^{k+1},\ldots,z_p^{k+1})\}$
  we deduce that the sequence $\{\mathcal{A}_1^*z_1^k\}$ is also bounded, which implies that
  the boundedness of $\big\{\|z_1^{k}\!-\!z_1^*\|_{\mathcal{A}_1\mathcal{A}_1^*+\mathcal{T}_{1}+\frac{2}{\sigma}\Sigma_1}^2\!\big\}$.
  Thus, there exists a subsequence $\{(z_1^k,\ldots,z_p^k,x^k)\}_{k\in K}$ that converges to a limit point,
  to say $(z_1^{\infty},\ldots,z_p^{\infty},x^{\infty})$. By part (a),
  $\{(\widetilde{z}_1^k,\ldots,\widetilde{z}_p^k)\}_{k\in K}$ also converges to
  $(z_1^{\infty},\ldots,z_p^{\infty})$.

  \medskip

  Next we argue that $(z_1^{\infty},\ldots,z_p^{\infty})$ is an optimal solution to problem \eqref{prob}
  and $x^{\infty}$ is the associated Lagrange multiplier.
  Since $\lim_{k\to+\infty}\|F(z_1^{k+1},\ldots,z_p^{k+1})\|=0$, we have
  \(
    \mathcal{A}_1^*z_1^{\infty}+\mathcal{A}_2^*z_2^{\infty}+\cdots+\mathcal{A}_p^*z_p^{\infty}-c=0.
  \)
  In addition, taking the limit $k\to\infty$ with $k\in K$ on the both sides of \eqref{opt-cond}
  and using the closedness of the graphs of $\partial\theta_i$ (see \cite{Roc70}),
  we have $-\mathcal{A}_ix^{\infty}\in\partial\theta_i(z_i^{\infty})$ for $i=1,\ldots,p$.
  The two sides and equation (\ref{optimal-cond}) imply that $(z_1^{\infty},\ldots,z_p^{\infty})$
  is an optimal solution of \eqref{prob} and $x^{\infty}$ is the associated Lagrange multiplier.

  \medskip

  To complete the proof of part (b), we only need to show that $(z_1^{\infty},\ldots,z_p^{\infty},x^{\infty})$
  is the unique limit point of $\{(z_1^k,\ldots,z_p^k,x^k)\}$. Recall that
  $(z_1^{\infty},\ldots,z_p^{\infty})$ is an optimal solution to \eqref{prob}
  and $x^{\infty}$ is the associated Lagrange multiplier. So, we could replace
  $(z_1^{*},\ldots,z_p^{*},x^*)$ with $(z_1^{\infty},\ldots,z_p^{\infty},x^{\infty})$
  in the previous arguments, starting from \eqref{opt-cond}. Thus,
  inequalities \eqref{theo-ineq31}-\eqref{theo-ineq32} still hold with $(z_1^*,\ldots,z_p^*,x^*)$
  replaced by $(z_1^\infty,\ldots,z_p^\infty,x^\infty)$, and then
  $\big\{\frac{\tau_k}{\alpha}\|\widetilde{w}^{k}-w^{\infty}\|_{\widetilde{\mathcal{G}}}^2
  +\frac{1}{\sigma^2}\|x^k-x^{\infty}\|^2+\tau_k\|z_1^{k}-z_1^{\infty}\|_{\mathcal{T}_{1}+\frac{2}{\sigma}\Sigma_1}^2\big\}$
  is convergent. Since this sequence is nonnegative and has a limit point $0$ for the subsequence
  $\{(\widetilde{z}_1^k,\ldots,\widetilde{z}_p^k,x^k)\}_{k\in K}$, we have
  \[
    \lim_{k\to\infty}\frac{\tau_k}{\alpha}\big\|\widetilde{w}^{k}-w^{\infty}\big\|_{\widetilde{\mathcal{G}}}^2
  +\frac{1}{\sigma^2}\big\|x^k-x^{\infty}\big\|^2+\tau_k\big\|z_1^{k}-z_1^{\infty}\big\|_{\mathcal{T}_{1}+\frac{2}{\sigma}\Sigma_1}^2
  =0.
  \]
  Moreover, $\lim_{k\to\infty}\|\mathcal{A}_1^*z_1^k\|=\|\mathcal{A}_1^*z_1^{\infty}\|$
  since $\lim_{k\to\infty}z_1^k = z_1^{\infty}$. By the results of part (a), $\tau_k>\overline{\tau}$
  and the positive definiteness of $\mathcal{A}_{1}\mathcal{A}_{1}^*+\mathcal{T}_{1}+\frac{2}{\sigma}\Sigma_1$,
  it follows that
  \begin{equation}\label{theo-equa36}
    \lim_{k\to\infty}z_i^k = z_i^{\infty}\quad{\rm for}\ \ i=1,2,\ldots,p\ \ {\rm and}\ \
    \lim_{k\to\infty}x^{k}=x^{\infty}.
  \end{equation}
   Thus, we show that $(z_1^{\infty},\ldots,z_p^{\infty},x^{\infty})$
  is the unique limit point of $\{(z_1^k,\ldots,z_p^k,x^k)\}$.

  \medskip
  \noindent
  {\bf Case 2: $\tau_k=\overline{\tau}$ for all $k\ge \overline{k}$} with $\overline{k}\in\mathbb{N}$.
  Now inequality (\ref{theo-result}) is specialized as
  \begin{align*}
   &\frac{1}{\alpha}\left(\big\|\widetilde{w}^{k}\!-\!w^*\big\|_{\widetilde{\mathcal{G}}}^2
                 -\big\|\widetilde{w}^{k+1}\!-\!w^*\big\|_{\widetilde{\mathcal{G}}}^2\right)
   +\frac{1}{\overline{\tau}\sigma^2}\left(\big\|x^k\!-\!x^*\big\|^2-\big\|x^{k+1}\!-\!x^*\big\|^2\right)\nonumber\\
  &\quad +\big\|z_1^{k}-z_1^*\big\|_{\mathcal{T}_{1}+\frac{2}{\sigma}\Sigma_1}^2
          -\big\|z_1^{k+1}-z_1^*\big\|_{\mathcal{T}_{1}+\frac{2}{\sigma}\Sigma_1}^2\nonumber\\
 &\ge\sum_{i=2}^{p-1}\big\|z_i^{k+1}-\widetilde{z}_i^k\big\|_{(1-\alpha)\mathcal{A}_i\mathcal{A}_i^*+2\mathcal{T}_{i}}^2
     +\big\|z_1^{k+1}-z_1^k\big\|_{\mathcal{T}_1}^2+2\big\|z_p^{k+1}-z_p^k\big\|_{\mathcal{T}_p}^2\nonumber\\
   &\quad +\big\|F(z_1^{k+1},\widetilde{z}_2^{k},\ldots,\widetilde{z}_p^{k})\big\|^2
           + (1-\overline{\tau})\big\|F(z_1^{k+1},z_2^{k+1},\ldots,z_p^{k+1})\big\|^2\nonumber\\
  &\qquad +\frac{2}{\sigma}\big\|z_1^{k}-z_1^*\big\|_{\Sigma_1}^2
         +\frac{2}{\sigma}\big\|z_p^{k}-z_p^*\big\|_{\Sigma_p}^2
         +\frac{2}{\sigma}\sum_{i=2}^{p-1}\big\|z_i^{k+1}\!-\!z_i^*\big\|_{\Sigma_i}^2.
  \end{align*}
  Since $\overline{\tau}\in(0,1)$, using the same arguments as those for Case 1, we have
  \begin{align}\label{theo-equa40}
   &\lim_{k\to+\infty}{\textstyle\sum_{i=2}^{p-1}\big\|z_i^{k+1}-\widetilde{z}_i^k\big\|_{(1-\alpha)\mathcal{A}_i\mathcal{A}_i^*+2\mathcal{T}_{i}}^2=0},\ \
     \lim_{k\to\infty}\|z_p^{k+1}-z_p^k\|_{\mathcal{T}_p}^2=0,\\
   &\lim_{k\to\infty}\big\|F(z_1^{k+1},\widetilde{z}_2^{k},\ldots,\widetilde{z}_p^{k})\big\|^2=0,\ \
     \lim_{k\to\infty}\big\|F(z_1^{k+1},z_2^{k+1},\ldots,z_p^{k+1})\big\|^2=0.
     \label{theo-equa41}
  \end{align}
  Combining the first limit in \eqref{theo-equa40} with the assumption of $\mathcal{T}_i$
  for $i=2,\ldots,p\!-\!1$, we have
  \begin{equation}\label{theo-equa42}
    \lim_{k\to+\infty}\|z_i^{k+1}\!-\widetilde{z}_i^{k}\|=0,\quad i=2,3,\ldots,p-1.
  \end{equation}
  From equations \eqref{theo-equa41} and \eqref{theo-equa42} and the second limit in \eqref{theo-equa40},
  we may deduce that
  \[
    \lim_{k\to\infty}\big\|z_p^{k+1}-z_p^k\big\|_{2\mathcal{T}_p+\mathcal{A}_p\mathcal{A}_p^*}=0.
  \]
  This, along with the assumption of $\mathcal{T}_p$, implies that $ \lim_{k\to+\infty}\|z_p^{k+1}\!-z_p^{k}\|=0$,
  while the second limit in \eqref{theo-equa41} implies that $ \lim_{k\to+\infty}\|x^{k+1}-x^{k}\|=0$.
  Thus, we complete the proof of part (a). Using the same arguments as those for Case 1 yields part (b).
  \end{proof}

  \medskip

  If all the linear operators $\mathcal{A}_i$ are surjective,
  then one can also obtain the conclusion of Theorem \ref{convergence1} by
  setting all $\mathcal{T}_i$ to be the zero operator in the proof of Theorem \ref{convergence1}.
 \begin{corollary}\label{convergence2}
  Suppose that Assumption \ref{assumpA} holds and the linear operators $\mathcal{A}_i$ for
  $i=1,2,\ldots,p$ are all surjective. Then, we have the following conclusions:
  \begin{itemize}
   \item[(a)]  $\lim\limits_{k \rightarrow +\infty}\big\|z_i^{k+1}- \widetilde{z}_i^k\big\| = 0$
                for $i=2,3,\ldots,p$ and $\lim\limits_{k \rightarrow +\infty}\big\|x^{k+1}-x^k\big\|=0$.

   \item[(b)] The sequences $\big\{(z_1^{k},\ldots,z_p^{k})\big\}$ and
              $\big\{(\widetilde{z}_1^{k},\ldots,\widetilde{z}_p^{k})\big\}$
              converge to an optimal solution to (\ref{prob}),
              and $\{x^{k}\}$ converges to an optimal solution to the dual problem of (\ref{prob}).
  \end{itemize}
 \end{corollary}

 \section{Applications to doubly nonnegative SDPs}

  Let $\mathcal{S}_{+}^n$ be the cone of $n\times n$ symmetric and positive semidefinite matrices
  in the space $\mathbb{S}^n$ of $n\times\! n$ symmetric matrices, which is endowed with the Frobenius
  inner product and its induced norm $\|\cdot\|$. The doubly nonnegative SDP problem takes the form of
  \begin{align}\label{PDNN-SDP}
   \max\Big\{-\!\big\langle C,X\big\rangle\ |\ \mathcal{A}_{E}X=b_{E},\ \mathcal{A}_IX\ge b_I,\ X\in\mathcal{S}_{+}^n,\ X-M\in\mathcal{K}\Big\},
  \end{align}
  where $b_{E}\in\mathbb{R}^{m_{E}},b_{I}\in\mathbb{R}^{m_{I}}$, and $X\!-\!M\in\mathcal{K}$ means that
  every entry of $X\!-\!M$ is nonnegative (of course, one can only require a subset of
  the entries of $X\!-\!M$ to be nonnegative or non-positive or free).
  An elementary calculation yields the dual of problem \eqref{PDNN-SDP} as
  \begin{align}\label{DDNN-SDP}
   &\min\ \big(\delta_{\mathbb{R}_{+}^{m_{I}}}(y_I)-\langle b_I,y_I\rangle\big)
     +\big(\delta_{\mathcal{K}^*}(Z)-\langle M,Z\rangle\big)-\langle b_E,y_E\rangle+\delta_{\mathcal{S}_{+}^n}(S)\nonumber\\
   &\ {\rm s.t.}\ \ \mathcal{A}_I^*y_I+Z+\mathcal{A}_E^*y_E+S=C,
  \end{align}
  where $\mathcal{K}^*$ is the positive dual cone of $\mathcal{K}$.
  Here we always assume that $\mathcal{A}_E$ is surjective. Clearly, problem \eqref{DDNN-SDP} takes
  the form of \eqref{prob} with $p=4$, and takes the form of \eqref{prob} with $p=3$ if
  the inequality constraint $\mathcal{A}_IX\ge b_I$ is removed. Hence, we can apply
  the proposed corrected ADMM with adaptive step-size for solving problem \eqref{DDNN-SDP}.

  \medskip

  For problem \eqref{DDNN-SDP}, instead of using the constraint qualification (CQ) in Assumption \ref{assumpA},
  we use the following more familiar Slater's CQ in the field of conic optimization.
  \begin{assumption}\label{assumpB}
  (a) For problem \eqref{PDNN-SDP}, there exists a point $\widehat{X}\in\mathbb{S}^n$ such that
      \[
        \mathcal{A}_{E}\widehat{X}=b_{E},\ \mathcal{A}_I\widehat{X}\ge b_I,\ \widehat{X}\in{\rm int}\,(\mathcal{S}_{+}^n),\ \widehat{X}\in\mathcal{K}.
      \]

  \noindent
  (b) For problem \eqref{DDNN-SDP}, there exists a point $(\widehat{y}_I,\widehat{Z},\widehat{y}_E,\widehat{S})
  \in\!\mathbb{R}^{m_I}\times\mathbb{S}^n\times\mathbb{R}^{m_E}\times\mathbb{S}^n$ such that
       \[
        \mathcal{A}_I^*\widehat{y}_I+\widehat{Z}+\mathcal{A}_{E}^*\widehat{y}_E+\widehat{S}=c,\ \widehat{S}\in{\rm int}\,(\mathcal{S}_{+}^n),\ \widehat{Z}\in\mathcal{K}^*,\
        \widehat{y}_I\in\mathbb{R}_{+}^{m_I}.
      \]
 \end{assumption}
  From \cite[Corollary 5.3.6]{BL06}, under Assumption \ref{assumpB},
  the strong duality for \eqref{PDNN-SDP} and  \eqref{DDNN-SDP} holds, and the following
  Karush-Kuhn-Tucker (KKT) condition has nonempty solutions:
  \begin{equation}\label{KKT-cond}
    \left\{\begin{array}{l}
     \mathcal{A}_{E}X-b_{E}=0,\\
     \mathcal{A}_I^*y_I+Z+\mathcal{A}_E^*y_E+S-C=0,\\
     \langle X,S\rangle=0,\ X\in\mathcal{S}_{+}^n,\ S\in\mathcal{S}_{+}^n,\\
     \langle X,Z\rangle =0,\ X\in\mathcal{K},\ Z\in\mathcal{K}^*,\\
     \langle y_I,\mathcal{A}_IX-b_I\rangle=0,\ \mathcal{A}_IX-b_I\ge 0,\ y_I\in\mathbb{R}_{+}^{m_I}.
     \end{array}\right.
  \end{equation}

   Let $\sigma\!>0$ be given. The augmented Lagrange function for \eqref{DDNN-SDP} is defined as follows
   \begin{align}
     L_{\sigma}(y_I,Z,y_E,S;X)&:=\delta_{\mathbb{R}_{+}^{m_{I}}}(y_I)-\langle b_I,y_I\rangle+(\delta_{\mathcal{K}^*}(Z)-\langle M,Z\rangle)
     -\langle b_E,y_E\rangle\nonumber\\
     &\qquad +\delta_{\mathcal{S}_{+}^n}(S) +\langle X,\mathcal{A}_I^*y_I+Z+\mathcal{A}_E^*y_E+S-C\rangle \nonumber\\
     &\qquad +\frac{\sigma}{2}\big\|\mathcal{A}_I^*y_I+Z+\mathcal{A}_E^*y_E+S-C\big\|^2\nonumber\\
    &\qquad \forall(y_I,Z,y_E,S,X)\in\mathbb{R}^{m_I}\times\mathbb{S}^n\times\mathbb{R}^{m_E}\times\mathbb{S}^n\times\mathbb{S}^n.\nonumber
   \end{align}
   Notice that the minimization of $L_{\sigma}(y_I,Z,y_E,S;X)$ with respect to variables $Z$ and $S$,
   respectively, have a closed form solution, while the minimization of $L_{\sigma}(y_I,Z,y_E,S;X)$
   with respect to variable $y_E$ is solvable since the operator $\mathcal{A}_E$ is assumed to be surjective.
   Hence, when applying the corrected semi-proximal ADMM for solving \eqref{DDNN-SDP},
   we do not introduce any proximal term to the three minimization problems.
   In addition, we adopt the solution order $y_I\rightarrow Z\rightarrow y_E\rightarrow S$
   for the subproblems involved in (S.1).
  By Remark \ref{Remark2.2}, such a solution order can guarantee that the hard constraints
  $y_I\in\mathbb{R}_{+}^{m_I}$ and $S\in\mathcal{S}_{+}^n$ are satisfied, and when the inequality
  constraint $\mathcal{A}_IX\ge b_I$ is removed, the hard constraints $Z\in\mathcal{K}^*$ and
  $S\in\mathcal{S}_{+}^n$ are satisfied. Extensive numerical tests indicate that
  such a solution order is the best one. Thus, we obtain the following algorithm.

  \bigskip

  \setlength{\fboxrule}{0.5pt}
  \noindent
  \fbox{
  \parbox{0.95\textwidth}
  {{\bf Algorithm\ CADMM: A corrected 4-block ADMM for solving \eqref{DDNN-SDP}}\

   \medskip
   \noindent
   Given parameters $\sigma>0,\alpha=0.999, \overline{\tau}=0.1$ and $\varepsilon=0.1$.
   Choose $\tau_0=1.95$ and a starting point $(y_I^0,Z^0,y_E^0,S^0,X^0)=
  (\widetilde{y}_I^0,\widetilde{Z}^0,\widetilde{y}_E^0,\widetilde{S}^0,X^0)
  \!\in\mathbb{R}_{+}^{m_I}\times\mathcal{K}^*\times\mathbb{R}^{m_E}\times\mathcal{S}_{+}^n\times\mathcal{S}_{+}^n$.
  Let $\mathcal{T}=\!\lambda_{\rm max}(\mathcal{A}_I\mathcal{A}_I^*)\mathcal{I}\!-\!\mathcal{A}_I\mathcal{A}_I^*$.
   For $k=0,1,\ldots$, perform the $k$th iteration as follows.
    \begin{description}
     \item[\bf Step 1.] Compute the following minimization problems
                        \begin{subnumcases}{}
                        y_I^{k+1}=\mathop{\arg\min}_{y_I\in\mathbb{R}_{+}^{m_I}}\ L_{\sigma}(y_I,\widetilde{Z}^k,\widetilde{y}_E^k,\widetilde{S}^k;X^k)
                                +\frac{\sigma}{2}\|y_I-\widetilde{y}_I^k\|_{\mathcal{T}}^2,\\
                        Z^{k+1}=\mathop{\arg\min}_{Z\in\mathcal{K}^*}\ L_{\sigma}(y_I^{k+1},Z,\widetilde{y}_E^k,\widetilde{S}^k;X^k),\\
                        y_E^{k+1}=\mathop{\arg\min}_{y_E\in\mathbb{R}^{m_E}}\ L_{\sigma}(y_I^{k+1},Z^{k+1},y_E,\widetilde{S}^k;X^k),\\
                        S^{k+1}=\mathop{\arg\min}_{S\in\mathcal{S}_{+}^n}\ L_{\sigma}(y_I^{k+1},Z^{k+1},y_E^{k+1},S;X^k).
                        \end{subnumcases}

     \item[\bf Step 2.] Let $X^{k+1}=X^{k}+\tau_{k}\sigma\big(\mathcal{A}_I^*y_I^{k+1}\!+\!Z^{k+1}
                                \!+\!\mathcal{A}_E^*y_E^{k+1}\!+\!S^{k+1}\!-\!C\big)$ where
                         \begin{equation*}
                            \tau_k:=\left\{\begin{array}{cl}\!
                                    \min(1+\delta_k,\tau_{k-1}) & {\rm if}\ 1+\delta_{k}>\overline{\tau}\\
                                     \overline{\tau}&{\rm otherwise}
                                     \end{array}\right.\ \ {\rm for}\ k\ge 1
                         \end{equation*}
                        \quad\ with
                         \begin{equation*}
                          \delta_{k}=\frac{\big\|\mathcal{A}_I^*y_I^{k+1}\!+\!\widetilde{Z}^{k}
                                \!+\!\mathcal{A}_E^*\widetilde{y}_E^{k}\!+\!\widetilde{S}^{k}\!-\!C\big\|^2\!
                                             -\varepsilon\big\|S^{k+1}\!-\!S^k\big\|^2}
                                              {\|\mathcal{A}_I^*y_I^{k+1}\!+\!Z^{k+1}
                                \!+\!\mathcal{A}_E^*y_E^{k+1}\!+\!S^{k+1}\!-\!C\|^2}-\varepsilon.
                         \end{equation*}

     \item[\bf Step 3.] Let $\widetilde{S}^{k+1}=S^{k+1},\widetilde{y}_I^{k+1} = y_I^{k+1},$ and
                 \begin{equation}\label{corr-step}
                  \left\{\begin{array}{l}
                  \widetilde{y}_E^{k+1}=\widetilde{y}_E^{k} + \alpha(y_E^{k+1}\!-\widetilde{y}_E^{k})
                                           -\big(\mathcal{A}_E\mathcal{A}_E^*\big)^{-1}(S^{k+1}-S^k),\\
                 \widetilde{Z}^{k+1} = \widetilde{Z}^{k}\! +\!\alpha(Z^{k+1}\!-\widetilde{Z}^{k}\!)
                                          -(S^{k+1}-S^k)-\mathcal{A}_{E}^*(\widetilde{y}_E^{k+1}\!-\!\widetilde{y}_E^{k}).
                  \end{array}\right.
                 \end{equation}
    \end{description}
   }
   }

  \subsection{Doubly nonnegative SDP problem sets}\label{subsec4.1}

  In our numerical experiments, we test the following five classes of doubly nonnegative
  SDP (DNN-SDP) problems, which can also be found in the literature \cite{WGY10,Toh04,STY14}.

  \medskip

  {\bf (i) SDP relaxation of BIQ problems}. It has been shown in \cite{Bur09} that
  under some mild assumptions, the following binary integer quadratic programming (BIQ) problem
  \[
     \min\Big\{\frac{1}{2}x^{\mathbb{T}}Qx +\langle c,x\rangle\ |\ x\in\{0,1\}^{n}\Big\}
  \]
  is equivalent to the completely positive programming (CPP) problem given by
  \[
    \min\Big\{\frac{1}{2}\langle Q,Y\rangle +\langle c,x\rangle\ |\ {\rm diag}(Y)-x=0,\
      X=[Y\ x;x^{\mathbb{T}}\ 1]\in\mathcal{C}_{n+1}\Big\},
  \]
  where $\mathcal{C}_{n+1}:=\{0\}\cup\big\{X\in\mathbb{S}^{n+1}\ |\
  X=\sum_{k\in K}z^k(z^k)^{\mathbb{T}}\ {\rm for\ some}\ \{z^k\}_{k\in K}\subset\mathbb{R}_{+}^{n+1}\backslash\{0\}\big\}$
  is the $(n+1)$-dimensional completely positive cone.
  It is well known that the CPP problem is intractable although $\mathcal{C}_{n+1}$ is convex.
  To solve the CPP problem, one would typically relax $\mathcal{C}_{n+1}$ to
  $\mathcal{S}_{+}^{n+1}\cap\mathcal{K}$, and obtain the following SDP relaxation problem
  \begin{align}\label{BIQ-prob}
   &\min\ \frac{1}{2}\langle Q,Y\rangle +\langle c,x\rangle\nonumber\\
   &\ {\rm s.t.}\ \ {\rm diag}(Y)-x=0,\ \ \alpha=1,\\
   &\qquad\  X=\left[\begin{matrix}
                   Y & x\\ x^{\mathbb{T}} & \alpha
                \end{matrix}\right]\in\mathcal{S}_{+}^{n+1},\ \ X\in\mathcal{K}\nonumber
  \end{align}
  where $\mathcal{K}=\big\{X\in\mathbb{S}^{n+1}\ |\ X\ge 0\big\}$ is the polyhedral cone.
  In our numerical experiments, the test data for the matrix $Q$ and the vector $c$
  are taken from Biq Mac Library maintained by Wiegele, which is available at
  http://biqmac.uni-klu.ac.at/biqmaclib.html

  \medskip

  {\bf (ii) $\theta_{+}$ problems}. This class of DNN-SDP problems arises from the relaxation of maximum
  stable set problems. Given a graph $G$ with edge set $E$, the SDP relaxation of the maximum stable
  set problem for the graph $G$ is given by
  \[
    \theta_{+}(G):=\max\left\{\langle ee^{\mathbb{T}},X\rangle\ |\ \langle \Xi_{ij},X\rangle=0\ \ (i,j)\in E,\
    \langle I,X\rangle =1,\ X\in\mathcal{S}_{+}^n,\ X\in\mathcal{K}\right\},
  \]
  where $e$ is the vector of ones with dimension known from the context, $\Xi_{ij}=e_ie_j^{\mathbb{T}}+e_je_i^{\mathbb{T}}$
  with $e_i$ denoting the $i$th column of the $n\times n$ identity matrix,
  and $\mathcal{K}=\big\{X\in\mathbb{S}^n\ |\ X\ge 0\big\}$.
  In our numerical experiments, we test the graph instances $G$ considered in \cite{Sloane05,TCCJ92,Toh04}.

   \medskip

  {\bf (iii) SDP relaxation of QAP problems}. Let $\mathbb{P}^n$ be the set of $n\times n$ permutation matrices.
  Given matrices $A,B\in\mathbb{S}^n$, the quadratic assignment problem is defined as
   \[
    \overline{v}_{\rm QAP}:=\min\Big\{\langle X,AXB\rangle\ |\ X\in\mathbb{P}^n\Big\}.
  \]
  We identify a matrix $X=[x_1\ x_2\ \ldots\ x_n]\in\mathbb{R}^{n\times n}$ with the $n^2$-vector $x=[x_1;\ldots,;x_n]$,
  and let $Y^{ij}$ be the $n\times n$ block corresponding to $x_ix_j^{\mathbb{T}}$ in the $n^2\times n^2$
  matrix $xx^{\mathbb{T}}$. It has been shown in \cite{PR09} that $\overline{v}_{\rm QAP}$ is bounded below
  by the number yielded by
  \begin{align}\label{QAP}
   v:=&\min\ \langle B\otimes A,Y\rangle\nonumber\\
   &\ {\rm s.t.}\ \ {\textstyle\sum_{i=1}^n}Y^{ii}=I,\nonumber\\
   &\qquad\ \langle I,Y^{ij}\rangle=\delta_{ij}\quad\forall 1\le i\le j\le n,\\
   &\qquad\  \langle \Gamma,Y^{ij}\rangle=1\quad \forall 1\le i\le j\le n,\nonumber\\
   &\qquad\  Y\in\mathcal{S}_{+}^{n^2},\ \ Y\in\mathcal{K},\nonumber
  \end{align}
  where $\Gamma$ is the matrix of ones,
  $\delta_{ij}=\!\left\{\begin{array}{ll}
                 1 & {\rm if}\ i=j\\
                 0 & {\rm if}\ i\ne j
                 \end{array}\right.$
  and $\mathcal{K}=\!\big\{X\in\mathbb{S}^{n^2}\ |\ X\ge 0\big\}$.
  In our numerical experiments, the test instances $(A,B)$ are taken from the QAP Library \cite{Hahn}.

   \medskip

  {\bf (iv) RCP problems}. This class of DNN-SDP problems arises from the SDP relaxation of clustering
  problems described in \cite[Eq(13)]{PW07} and takes the following form
  \begin{equation}\label{RCP}
   \min\Big\{\big\langle W,I\!-\!X\big\rangle\ |\ Xe=e,\ \langle I,X\rangle=\kappa,\ X\in\mathcal{S}_{+}^n,\ X\in \mathcal{K}\Big\},
  \end{equation}
  where $W$ is the so-called affinity matrix whose entries represent the similarities of the objects
  in the dataset, $e$ is the vector of ones, $\kappa$ is the number of clusters, and
  $\mathcal{K}$ is the cone $\big\{X\in\mathbb{S}^n\ |\ X\ge 0\big\}$.
  All the data sets we test are from the UCI Machine Learning Repository (available at http://archive.ics.uci.edu/ml/datasets.html).
  For some large size data sets, we only select the first $n$ rows. For example,
  the original data set ``spambase'' has $4061$ rows and we select the first $1500$ rows to obtain
  the test problem ``spambase-large.2'' for which the number ``2'' means that there are $\kappa=2$ clusters.

  \medskip

  {\bf (v) SDP relaxation of FAP problems}. Let $G=(V,E)$ be an undirected graph with vertex set
  $V$ and edge set $E\in V\times V$, and $W$ be a weight matrix for $G$ such that $W_{ij}=W_{ji}$ is
  the weight associated with $(i,j)\in E$. For those edges $(i,j)\not\in E$,
  we assume $W_{ij}=W_{ji}=0$. Let $U\subseteq E$ be a given edge subset. 
  This class of problems has the form
  \begin{align}\label{FAP}
   &\max \Big\langle \frac{\kappa-1}{2\kappa}L(G,W)-\frac{1}{2}{\rm Diag}(We),X\Big\rangle\nonumber\\
   &\ {\rm s.t.}\ \ {\rm diag}(X)=e,\ \ X\in\mathcal{S}_{+}^n,\\
   &\qquad\  \langle -\Xi_{ij},X\rangle={2}/{(\kappa-1)}\quad \forall(i,j)\in U\subseteq E,\nonumber\\
   &\qquad\  \langle -\Xi_{ij},X\rangle\le {2}/{(\kappa-1)}\quad\forall(i,j)\in E\backslash U,\nonumber
  \end{align}
  where $\kappa>1$ is an integer, $L(G,W):={\rm Diag}(We)-W$ is the Laplacian matrix,
  $\Xi_{ij}$ and $e$ are same as above. Let $M_{ij}=-\frac{1}{\kappa-1}$ if $(i,j)\in E$ 
  and otherwise $M_{ij}=0$. Then (\ref{FAP}) is also equivalent to
  \begin{align}\label{EFAP}
  &\max\ \Big\langle \frac{\kappa-1}{2\kappa}L(G,W)-\frac{1}{2}{\rm Diag}(We),X\Big\rangle\nonumber\\
   &\ {\rm s.t.}\ \ {\rm diag}(X)=e,\ \ X\in\mathcal{S}_{+}^n,\ X-M\in\mathcal{K},
  \end{align}
  where
   \(
    \mathcal{K}=\big\{X\in\mathbb{S}^n\ |\ X_{ij}=0\ \ \forall(i,j)\in U,\ X_{ij}\ge 0\ \ \forall(i,j)\in E\backslash U\big\}
  \)
  (see \cite{STY14}).

  \medskip

  The above five classes of DNN-SDP problems come from the SDP relaxation for
  some difficult combinatorial optimization problems. For these problems, one usually adds some additional
  valid inequalities so as to obtain tighter bound for the original combinatorial optimization problems.
  For example, to obtain a tighter bound for the BIQ problems, one can add four classes of
  valid inequalities to \eqref{BIQ-prob} and get the following problems:
  \begin{align}\label{EBIQ}
     &\min\ \frac{1}{2}\langle Q,Y\rangle +\langle c,x\rangle\nonumber\\
     &\ {\rm s.t.}\ \ {\rm diag}(Y)-x=0,\ \ \alpha=1,\nonumber\\
     &\qquad\  X=\left[\begin{matrix}
                     Y & x\\ x^{\mathbb{T}} & \alpha
                  \end{matrix}\right]\in\mathcal{S}_{+}^{n+1},\ \ X\in\mathcal{K},\nonumber\\
     &\qquad\ -Y_{ij}\!+\!x_i\ge 0,\ -Y_{ij}\!+\!x_j\ge 0,\ Y_{ij}\!-\!x_i\!-\!x_j\ge -1,\ \forall i<j,\ j=2,\ldots,n\!-\!1,\nonumber\\
     &\qquad\quad\ Y_{ij}+Y_{ik}+Y_{jk}-x_i-x_j-x_k \ge -1,\ \ j\ne i,\ k\ne i,\ k\ne j
  \end{align}
  where $\mathcal{K}=\big\{X\in\mathbb{S}^{n+1}\ |\ X\ge 0\big\}$, and the set of the first three 
  inequalities are obtained from the valid inequalities
  $x_i(1-x_j)\ge 0, x_j(1-x_i)\ge 0,(1-x_i)(1-x_j)\ge 0$ when $x_i,x_j$ are binary variables.
  In the sequel, we call \eqref{EBIQ} the extended BIQ problem.

  \subsection{Numerical results for DNN-SDPs without $\mathcal{A}_IX\ge b_I$ constraints}

  In this subsection, we apply CADMM for solving the doubly nonnegative SDP problems without
  inequality constraints $\mathcal{A}_IX\ge b_I$ described in last subsection,
  and compare its performance with that of the $3$-block ADMM with step-size $\tau=1.618$
  and the $3$-block ADMM with Guassian back substitution proposed in \cite{HTY12}.
  We call the last two methods ADMM3d and ADMM3g, respectively. We have implemented CADMM,
  ADMM3d and ADMM3g in MATLAB, where the correction step-size of ADMM3g was set to
  be $0.999$ instead of $1$ as in \cite{HTY12} for the convergence guarantee.
  Among others, the solution order of subproblems involved in the $3$-block ADMM
  and the prediction step of ADMM3g is same as that of subproblems in (S.1) of CADMM.
  Extensive numerical tests show that this order is also the best for ADMM3d and ADMM3g.
  Notice that ADMM3d here is different from the ADMM developed by Wen et al. \cite{WGY10}
  since the latter uses the solution order of $y\rightarrow Z\rightarrow S$.
  The computational results for all the DNN-SDP problems
  are obtained on a Windows system with Intel(R) Core(TM) i3-2120 CPU@3.30GHz.

  \medskip

  We measure the accuracy of an approximate optimal solution $(X,y_E,S,Z)$ for \eqref{PDNN-SDP}
  and \eqref{DDNN-SDP} by using the relative residual
  \(
    \eta=\max\big\{\eta_{P},\eta_{D},\eta_{\mathcal{S}},\eta_{\mathcal{K}},\eta_{\mathcal{S}^*},\eta_{\mathcal{K}^*},\eta_{C_1},\eta_{C_2}\big\}
  \)
  where
  \begin{align*}
   \eta_{P}\!=\!\frac{\|\mathcal{A}_EX\!-\!b_E\|}{1+\|b_E\|},\ \eta_D\!=\!\frac{\|\mathcal{A}_E^*y_E\!+\!S\!+\!Z\!-\!C\|}{1+\|C\|},\
   \eta_{\mathcal{S}}\!=\!\frac{\|\Pi_{\mathcal{S}_{+}^n}(-X)\|}{1+\|X\|},\ \eta_{\mathcal{K}}=\frac{\|\Pi_{\mathcal{K}^*}(-X)\|}{1+\|X\|},\\
   \eta_{\mathcal{S}^*}\!=\!\frac{\|\Pi_{\mathcal{S}_{+}^n}(-S)\|}{1+\|S\|},\
   \eta_{\mathcal{K}^*}\!=\!\frac{\|\Pi_{\mathcal{K}^*}(-Z)\|}{1+\|Z\|},\
   \eta_{C_1}\!=\!\frac{\langle X,S\rangle}{1+\|X\|+\|S\|},\
   \eta_{C_2}\!=\!\frac{\langle X,Z\rangle}{1+\|X\|+\|Z\|}.
  \end{align*}
  In addition, we also compute the relative gap by
  \(
    \eta_g = \frac{\langle C,X\rangle-\langle b_E,y_E\rangle}{1+|\langle C,X\rangle + |\langle b_E,y_E\rangle|}.
  \)
  We terminated the solvers CADMM, ADMM3g and ADMM3d whenever $\eta<10^{-6}$ or
  the number of iteration is over the maximum number of iterations $k_{\rm max}=20000$.

  \medskip

  In the implementation of all the solvers, the penalty parameter $\sigma$ is dynamically adjusted
  according to the progress of the algorithms. The exact details on the adjustment strategies are
  too tedious to be presented here but it suffices to mention that the key idea to adjust $\sigma$
  is to balance the progress of primal feasibilities $(\eta_P,\eta_{\mathcal{S}},\eta_{\mathcal{K}})$
  and dual feasibilities $(\eta_D,\eta_{\mathcal{S}^*},\eta_{\mathcal{K}^*})$. In addition,
  all the solvers also adopt some kind of restart strategies to ameliorate slow convergence.
  During the numerical tests, we use the same adjustment strategy of $\sigma$ and restart strategy
  for all the solvers.

   \medskip

   Table \ref{tab1} reports the number of problems that are successfully solved to the accuracy
   of $10^{-6}$ in $\eta$ by each of the three solvers within the maximum number of iterations.
   We see that CADMM and ADMM3d solved successfully all instances from ${\bf BIQ}$, ${\bf RCP}$,
   ${\bf FAP}$ and $\theta_{+}$, and for ${\bf QAP}$ problems CADMM and ADMM3d solved successfully
   ${\bf 39}$ and ${\bf 35}$, respectively; while ADMM3g solved successfully all instances from ${\bf RCP}$
   and ${\bf FAP}$, but failed to ${\bf 1}$ tested problem from ${\bf BIQ}$, ${\bf 5}$ tested problems
   from $\theta_{+}$ and ${\bf 58}$ tested problems from ${\bf QAP}$.
   That is, CADMM solved the most number of instances to the required accuracy,
   with ADMM3d in the second place, followed by ADMM3g.

  \begin{table}[h]
  \center
  \caption{Numbers of problems that are solved to the accuracy of $10^{-6}$ in $\eta$}
  \begin{tabular}{|c|c|c|c|}
   \hline
   \diagbox{\bf Problem\ set}{\bf Solvers}& ${\bf CADMM}$ & ${\bf ADMM3d}$ & ${\bf ADMM3g}$ \\
   \hline
    ${\bf BIQ}(165)$ & $165$  & $165$   & $164$ \\
   \hline
   $\theta_{+}(113)$ & $113$  & 113 & 108 \\
   \hline
   ${\bf QAP}(95)$ & $39$  & 35  & 37\\
   \hline
   ${\bf RCP}(120)$ & $120$  & $120$  & 120 \\
   \hline
   ${\bf FAP}(13)$ & $13$  & $13$ & 13 \\
   \hline
   ${\bf Total}(506)$ & $450$  & 446 & 443 \\
   \hline
  \end{tabular}
  \label{tab1}
  \end{table}

  \medskip

  Table \ref{tab2} reports the detailed numerical results of CADMM, ADMM3d and ADMM3g in solving
  all test instances. From this table, one can learn that CADMM requires the fewest iterations
  for about $69\%$ test problems though the computing time is comparable even a little
  more than that of the ADMM3d due to some additional computation cost in the correction step,
  while ADMM3g requires the most iterations for most of problems and
  at least $1.5$ times as many iterations as CADMM does for about $30\%$ test problems.

  \medskip

  Figure \ref{BIQ-fig} (respectively, Figure \ref{RCP-fig}) shows the performance profiles of CADMM,
  ADMM3d and ADMM3g in terms of number of iterations and computing time, respectively, for the total
  $165$ ${\bf BIQ}$ (respectively, $120$ ${\bf RCP}$) tested problems.
  We recall that a point $(x,y)$ is in the performance profiles curve of a method
  if and only if it can solve $(100y)\%$ of all tested problems no slower than $x$ times of
  any other methods. It can be seen that CADMM requires the least number of iterations for
  at least $90\%$ ${\bf BIQ}$ tested problems and $75\%$ ${\bf RCP}$ tested problems,
  and its computing time is at most $1.3$ times as many as that of the fastest solver
  for $90\%$ instances; while ADMM3g requires the most number of iterations for almost
  all test instances, and for about $20\%$ BIQ tested problems,
  its number of iterations is at least twice as many as that of the best solver.

  \begin{figure}[h]
 \begin{minipage}{0.5\linewidth}
 \centering
 \includegraphics[width=3.1in]{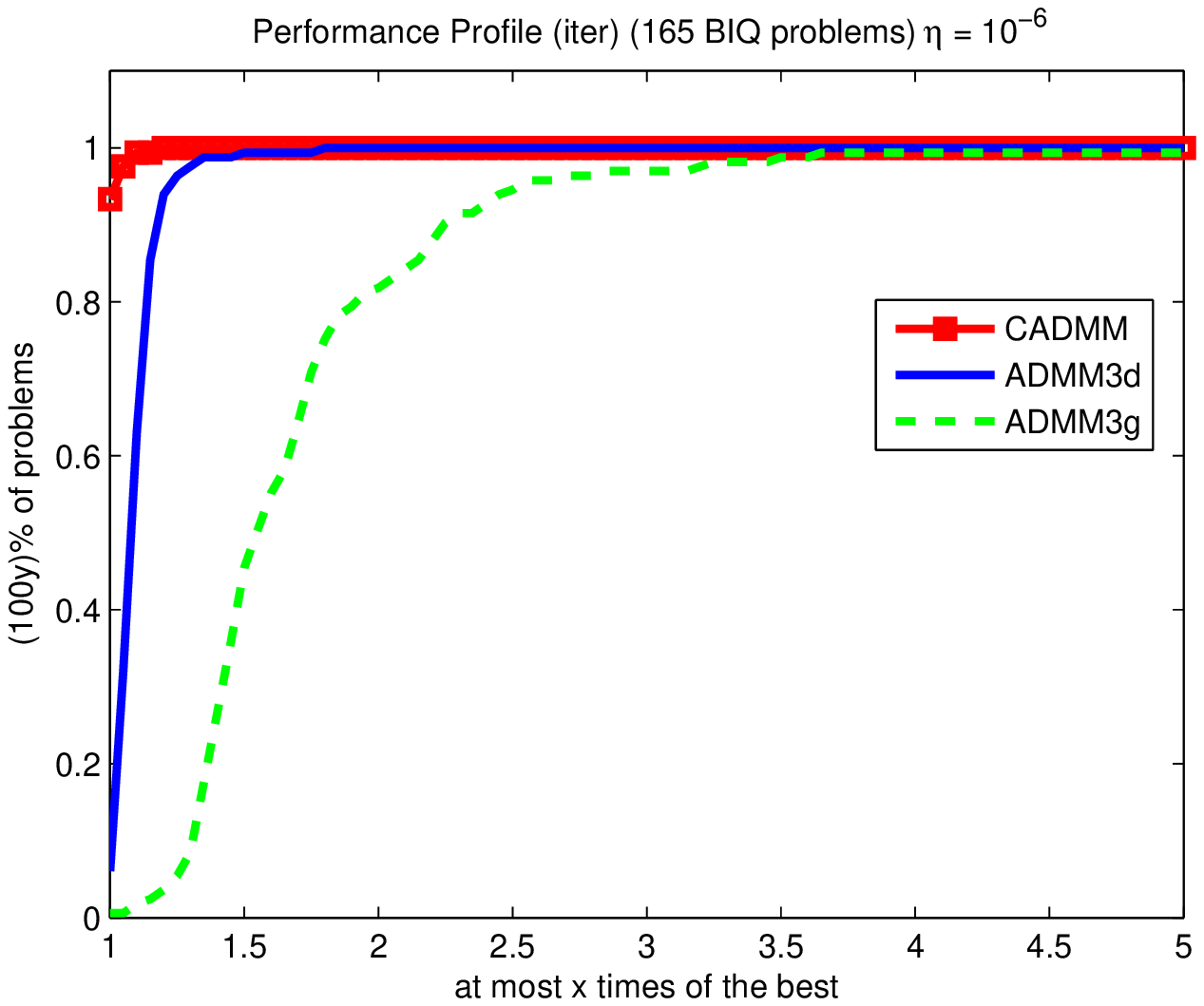}
 \end{minipage}%
 \begin{minipage}{0.5\linewidth}
 \centering
 \includegraphics[width=3.1in]{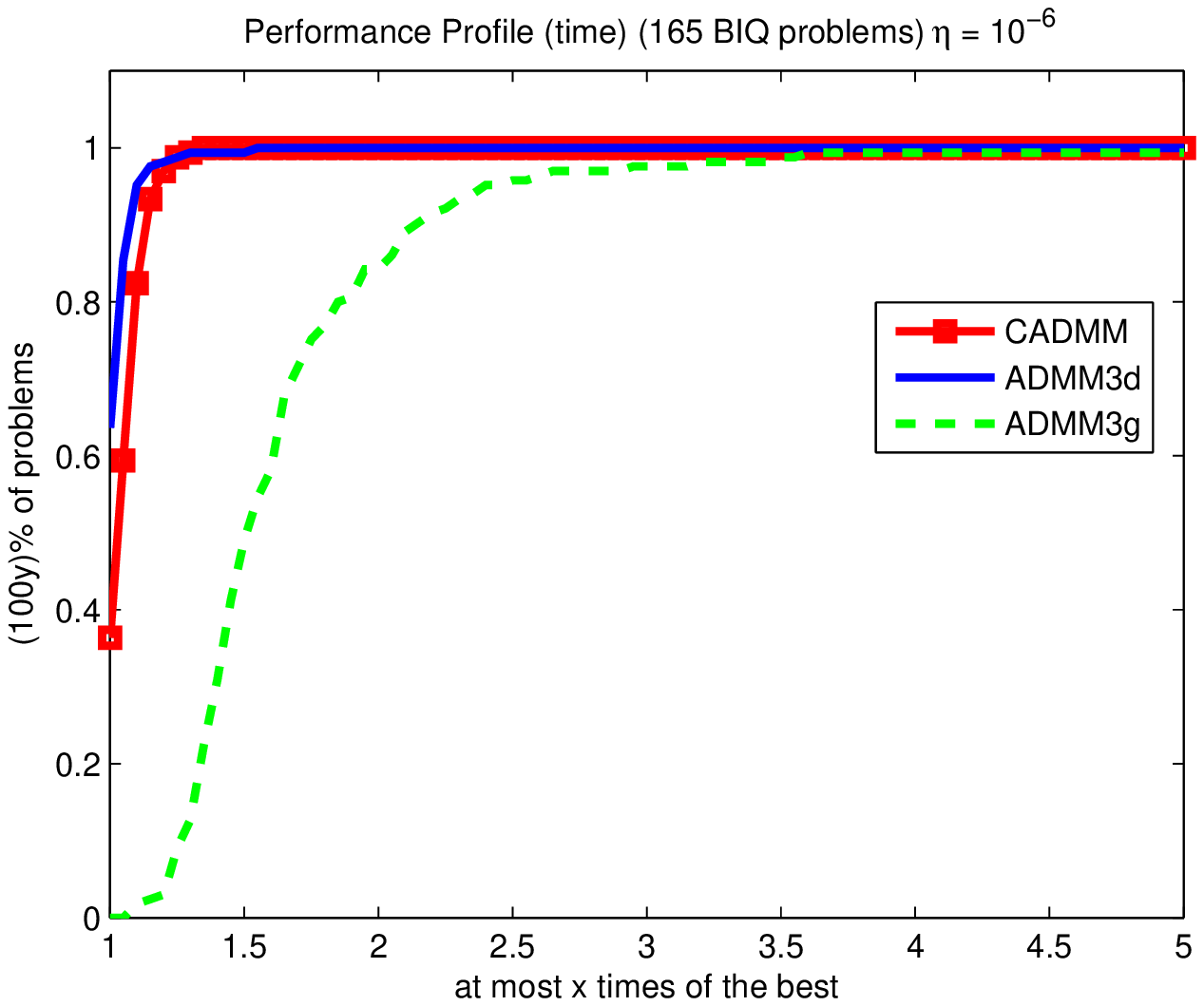}
 \end{minipage}
  \caption{\small Performance profiles of the number of iterations and computing time for BIQ}
 \label{BIQ-fig}
 \end{figure}
 \begin{figure}[h]
 \begin{minipage}{0.5\linewidth}
 \centering
 \includegraphics[width=3.1in]{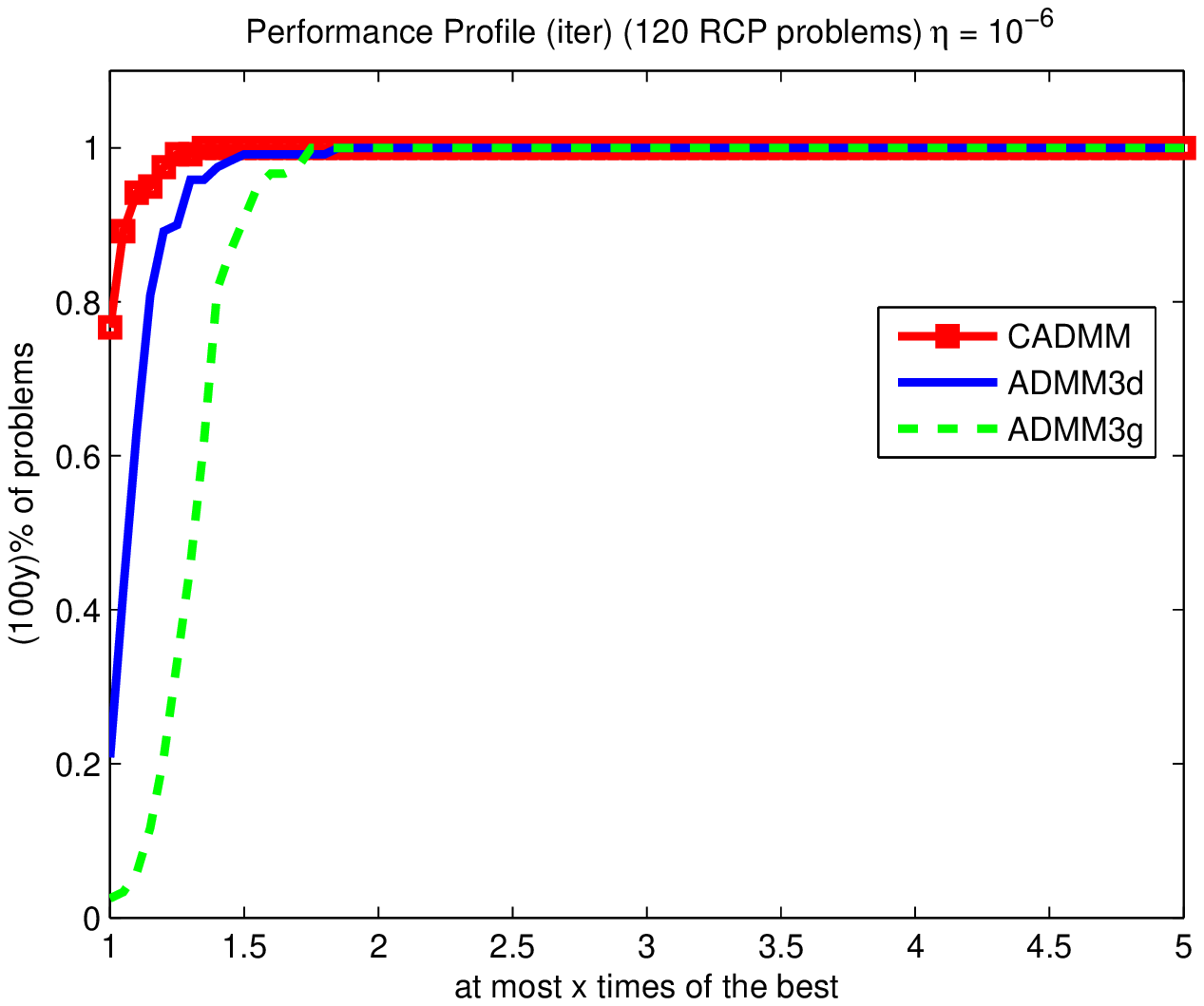}
 \end{minipage}%
 \begin{minipage}{0.5\linewidth}
 \centering
 \includegraphics[width=3.1in]{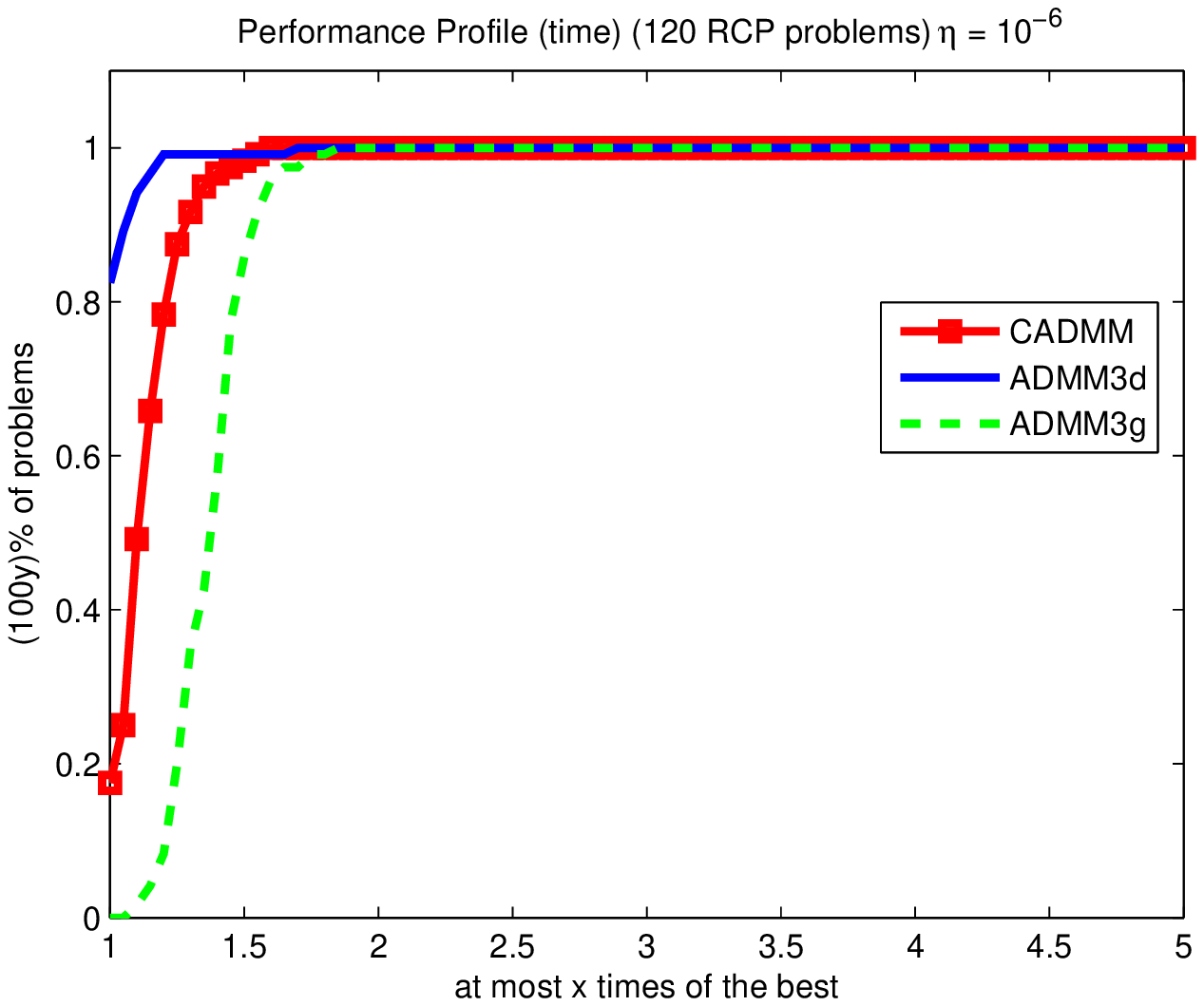}
 \end{minipage}
 \caption{\small Performance profiles of the number of iterations and computing time for RCP}
 \label{RCP-fig}
\end{figure}
\begin{figure}[h]
 \begin{minipage}{0.5\linewidth}
 \centering
 \includegraphics[width=3.1in]{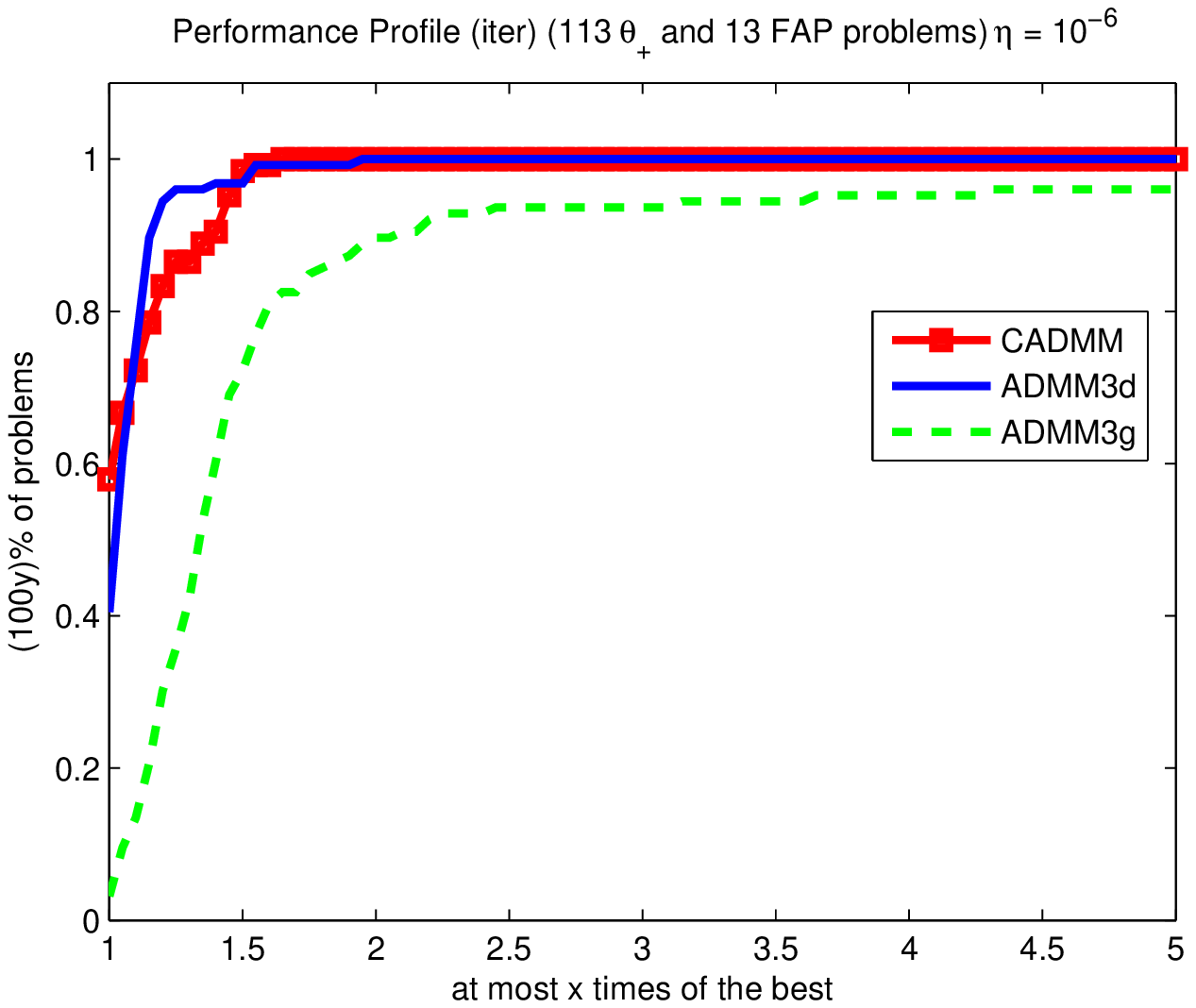}
 \end{minipage}%
 \begin{minipage}{0.5\linewidth}
 \centering
 \includegraphics[width=3.1in]{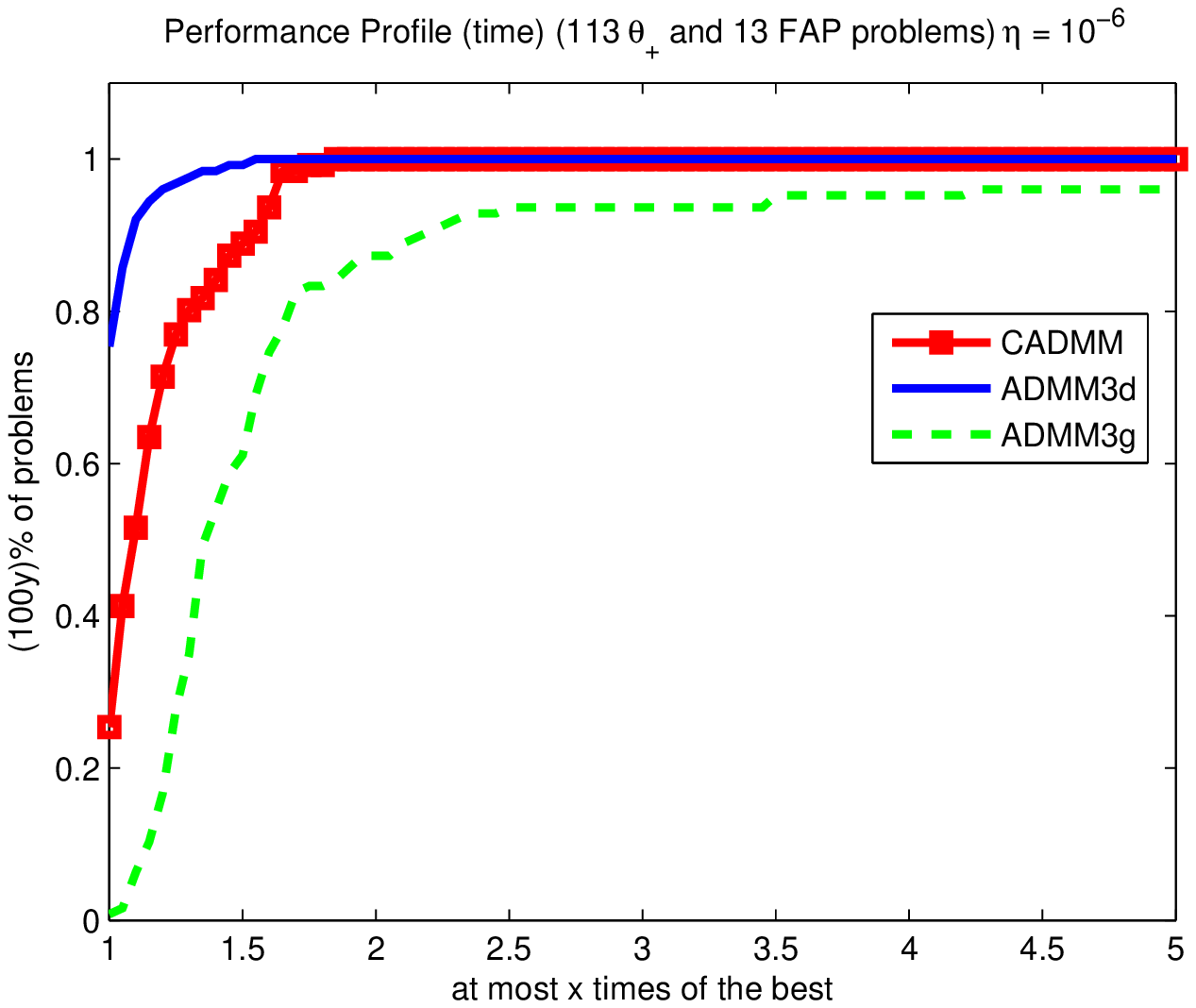}
 \end{minipage}
 \caption{\small Performance profiles of the number of iterations and computing time for $\theta_{+}$ and FAP}
 \label{theta-fig}
\end{figure}
 \begin{figure}[htbp]
 \begin{minipage}{0.5\linewidth}
 \centering
 \includegraphics[width=3.1in]{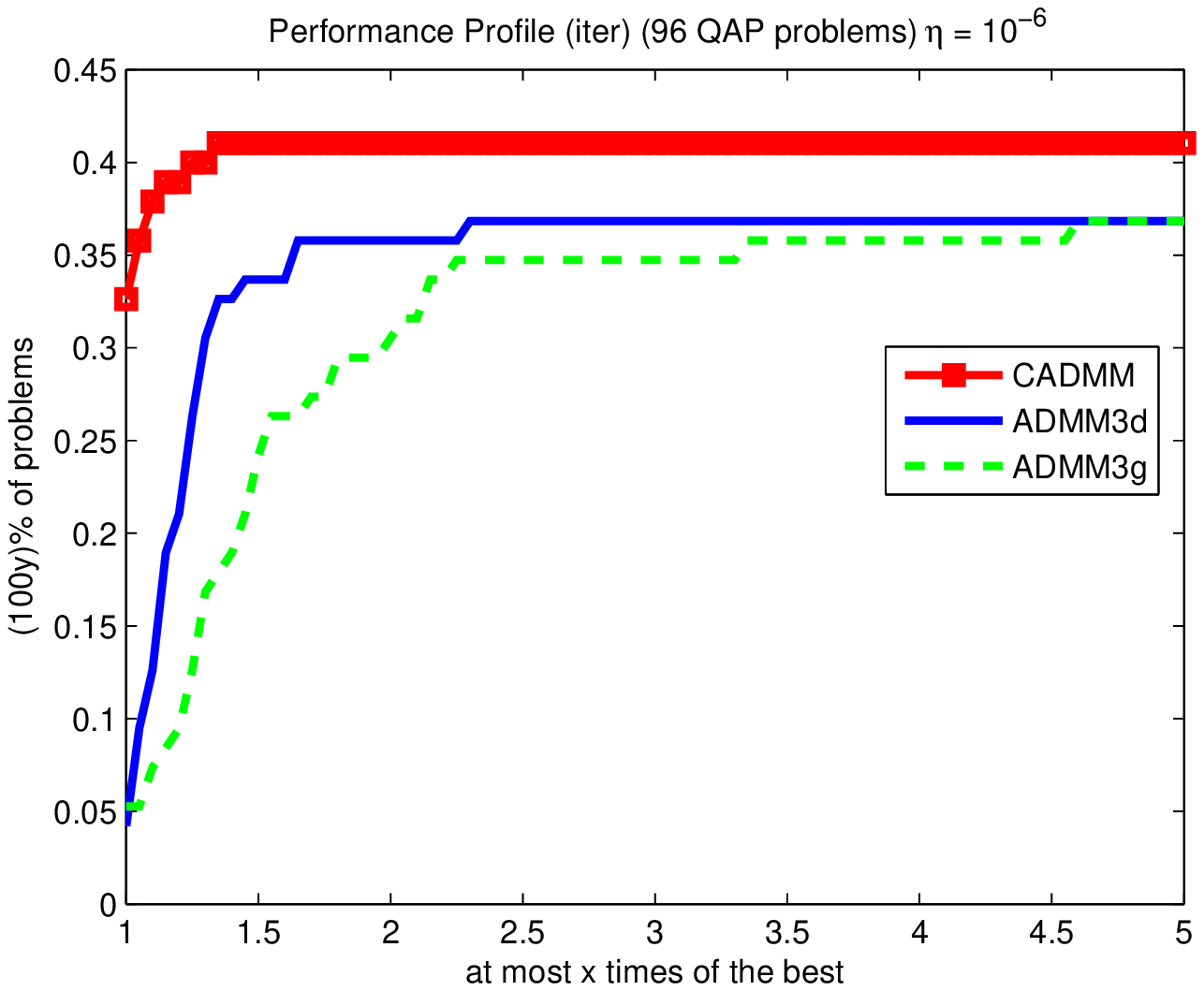}
 \end{minipage}%
 \begin{minipage}{0.5\linewidth}
 \centering
 \includegraphics[width=3.1in]{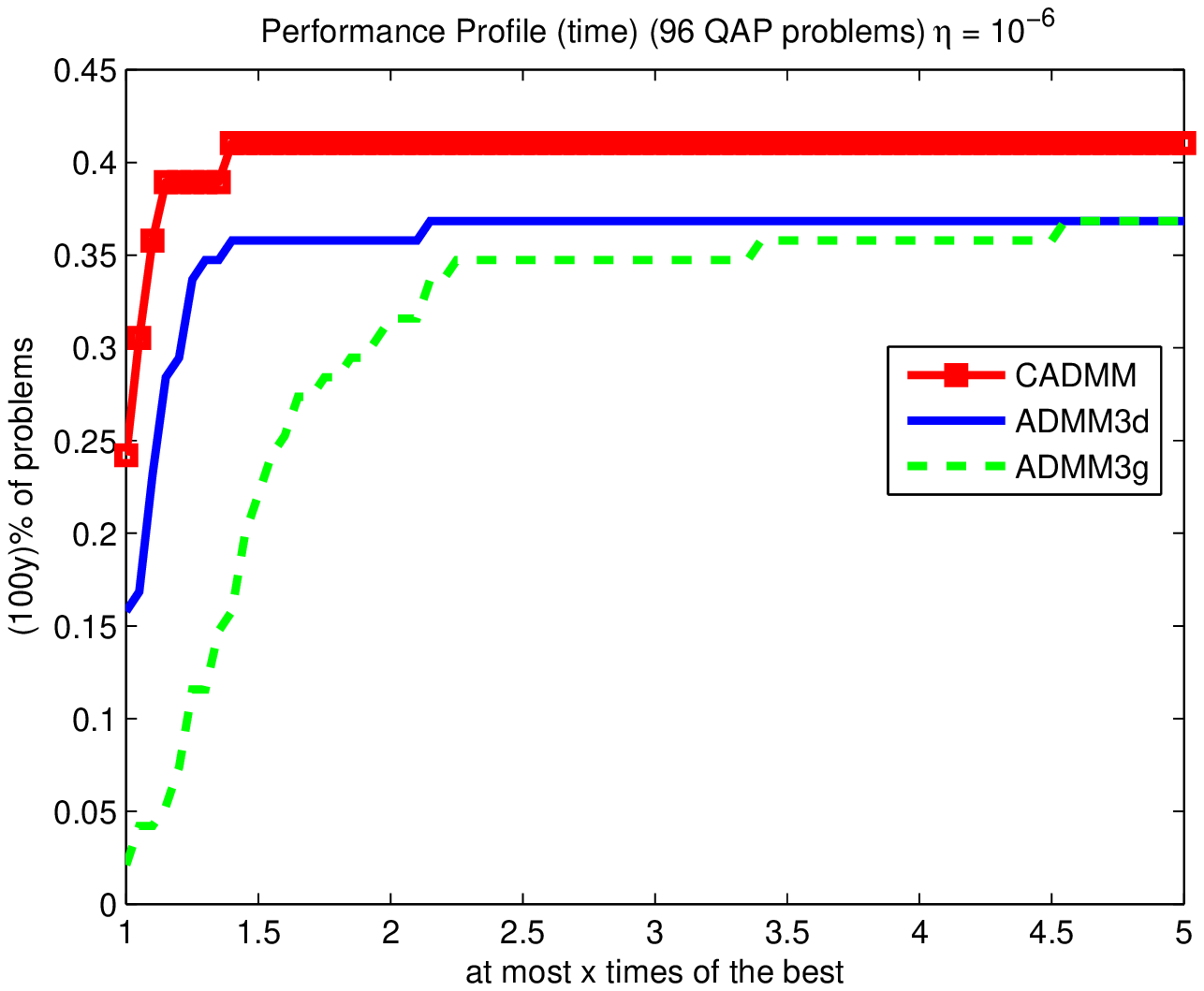}
 \end{minipage}
 \caption{\small Performance profiles of the number of iterations and computing time for QAP}
 \label{QAP-fig}
\end{figure}

 \medskip

  Figure \ref{theta-fig} (respectively, Figure \ref{QAP-fig}) shows the performance profiles of CADMM, ADMM3d and ADMM3g
  in terms of number of iterations and computing time, respectively, for the total $113$ $\theta_{+}$
  and $13$ ${\bf FAP}$ (respectively, $95$ ${\bf QAP}$) tested problems.
  One can see from Figure \ref{theta-fig} that CADMM requires the comparable iterations as ADMM3d does
  for the $\theta_{+}$ and ${\bf FAP}$ tested problems, for which the former requires the least number of
  iterations for about $60\%$ problems and the latter requires the least number of iterations for
  at most $40\%$ problems, while ADMM3g requires at least $1.5$ times as many iterations as
  the best solvers for about $30\%$ problems. Figure \ref{QAP-fig} indicates that
  for the ${\bf QAP}$ tested problems, which are the most difficult among the five classes,
  CADMM has remarkable superiority to ADMM3d and ADMM3g in terms of iterations and computing time,
  and CADMM requires the least number of iterations for more than $30\%$ problems,
  while ADMM3d and  ADMM3g need the least number of iterations only for $5\%$ problems.

  \subsection{Numerical results for DNN-SDPs with $\mathcal{A}_IX\ge b_I$ constraints}

  We apply CADMM for solving the extended BIQ problems described in \eqref{EBIQ},
  and compare its performance with the linearized ADMMG in \cite{HY13} (we call the method LADMM4g and
  use the parameter $\alpha=0.999$ in the Gaussian back substitution step). Notice that one may apply
  the directly extended ADMM with $4$ blocks (although without convergent guarantee) for solving \eqref{EBIQ}
  by adding a proximal term $\frac{\sigma}{2}\|y_I-y_I^k\|_{\mathcal{T}}^2$ for the $y_I$ part,
  where $\mathcal{T}=\|\mathcal{A}_I\mathcal{A}_I^*\|\mathcal{I}\!-\!\mathcal{A}_I\mathcal{A}_I^*$.
  We call this method ADMM4d, and compare the performance of CADMM with that of ADMM4d with $\tau=1.618$.
  The computational results for all the extended BIQ problems are obtained on
  the same desktop computer as before.

  \medskip

  We measure the accuracy of an approximate optimal solution $(X,y_I,Z,y_E,S)$ for \eqref{PDNN-SDP}
  and \eqref{DDNN-SDP} by the relative residual
  \(
   \eta=\max\big\{\eta_{P},\eta_{D},\eta_{\mathcal{S}},\eta_{\mathcal{K}},\eta_{\mathcal{S}^*},\eta_{\mathcal{K}^*},\eta_{C_1},\eta_{C_2},\eta_I,\eta_{I^*}\big\},
  \)
  where $\eta_{P},\eta_{\mathcal{S}},\eta_{\mathcal{K}},\eta_{\mathcal{S}^*},\eta_{\mathcal{K}^*},\eta_{C_1},\eta_{C_2}$
  are defined as before, and $\eta_D,\eta_I,\eta_{I^*}$ are given by
  \begin{align*}
   \eta_D\!=\!\frac{\|\mathcal{A}_I^*y_I\!+Z+\mathcal{A}_E^*y_E+\!S-\!C\|}{1+\|C\|},\
   \eta_{I}\!=\!\frac{\|\max(0,b_I-\mathcal{A}_IX)\|}{1+\|b_I\|},\ \eta_{I^*}\!=\!\frac{\|\max(0,-y_I)\|}{1+\|y_I\|}.
  \end{align*}
  We also compute the relative gap by
  \(
    \eta_g = \frac{\langle C,X\rangle-(\langle b_E,y_E\rangle+\langle b_I,y_I\rangle)}{1+|\langle C,X\rangle + |\langle b_E,y_E\rangle+\langle b_I,y_I\rangle|}.
  \)
  The solvers CADMM, ADMM4d and LADMM4g were terminated whenever $\eta<10^{-6}$ or
  the number of iteration is over the maximum number of iterations $k_{\rm max}=40000$.

  \medskip

  Table \ref{tab3} reports the detailed numerical results for the solvers CADMM, ADMM4d and LADMM4g
  in solving a collection of ${\bf 165}$ extended BIQ problems. Figure \ref{EBIQ-fig} shows the performance 
  profiles of CADMM, ADMM4d and LADMM4g in terms of the number of iterations and the computing time, 
  respectively, for the total ${\bf 165}$ extended BIQ tested problems. It can be seen that CADMM requires 
  the least number of iterations for $80\%$ tested problems although its computing time is comparable with 
  that of ADMM4d, which requires the least computing time for $90\%$ tested problems,
  while ADMM4g requires $1.5$ times as many as iterations as CADMM does for $73\%$ tested problems.

 \begin{figure}[htbp]
   \begin{minipage}{0.5\linewidth}
   \centering
   \includegraphics[width=3.1in]{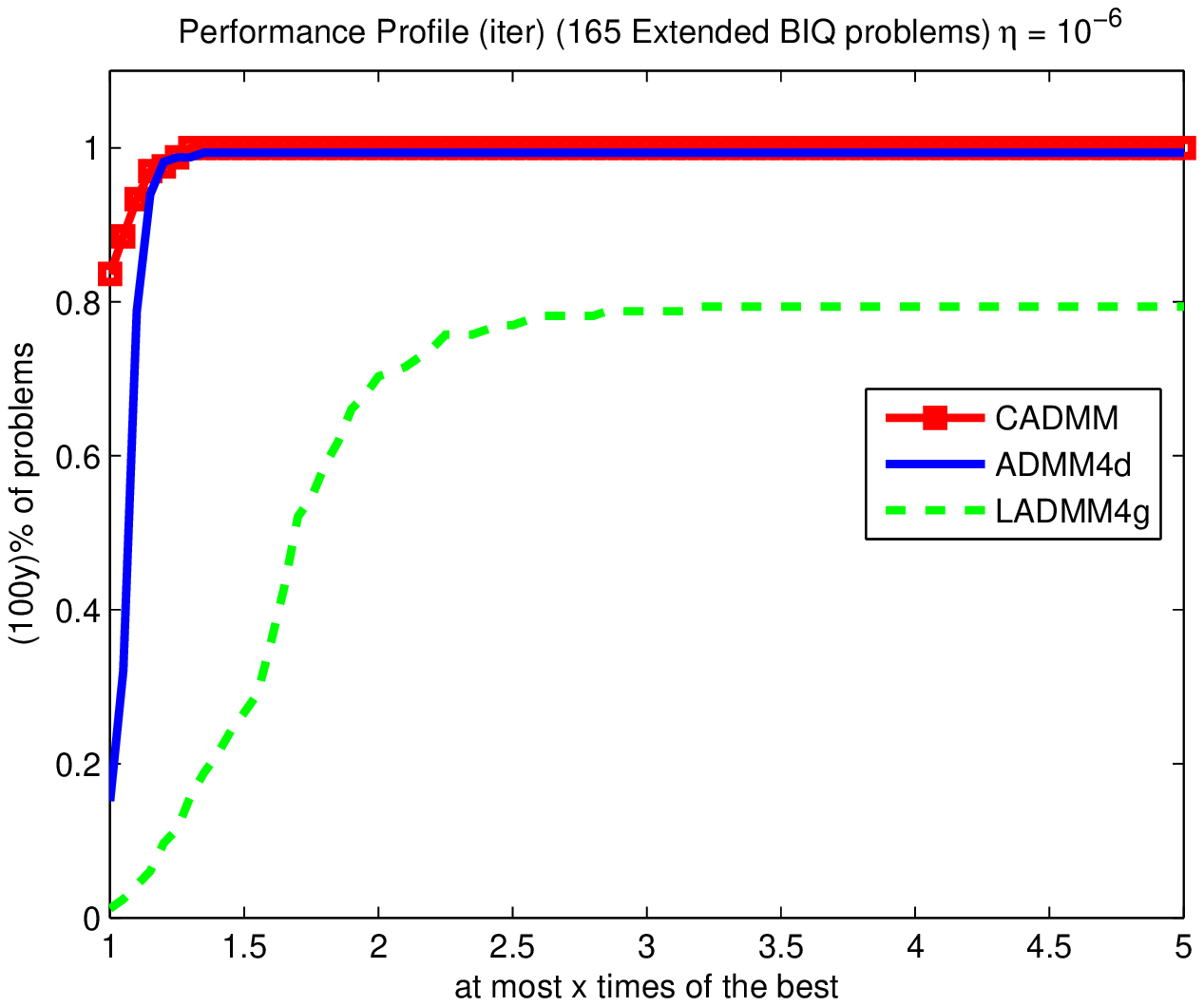}
   \end{minipage}%
   \begin{minipage}{0.5\linewidth}
   \centering
   \includegraphics[width=3.1in]{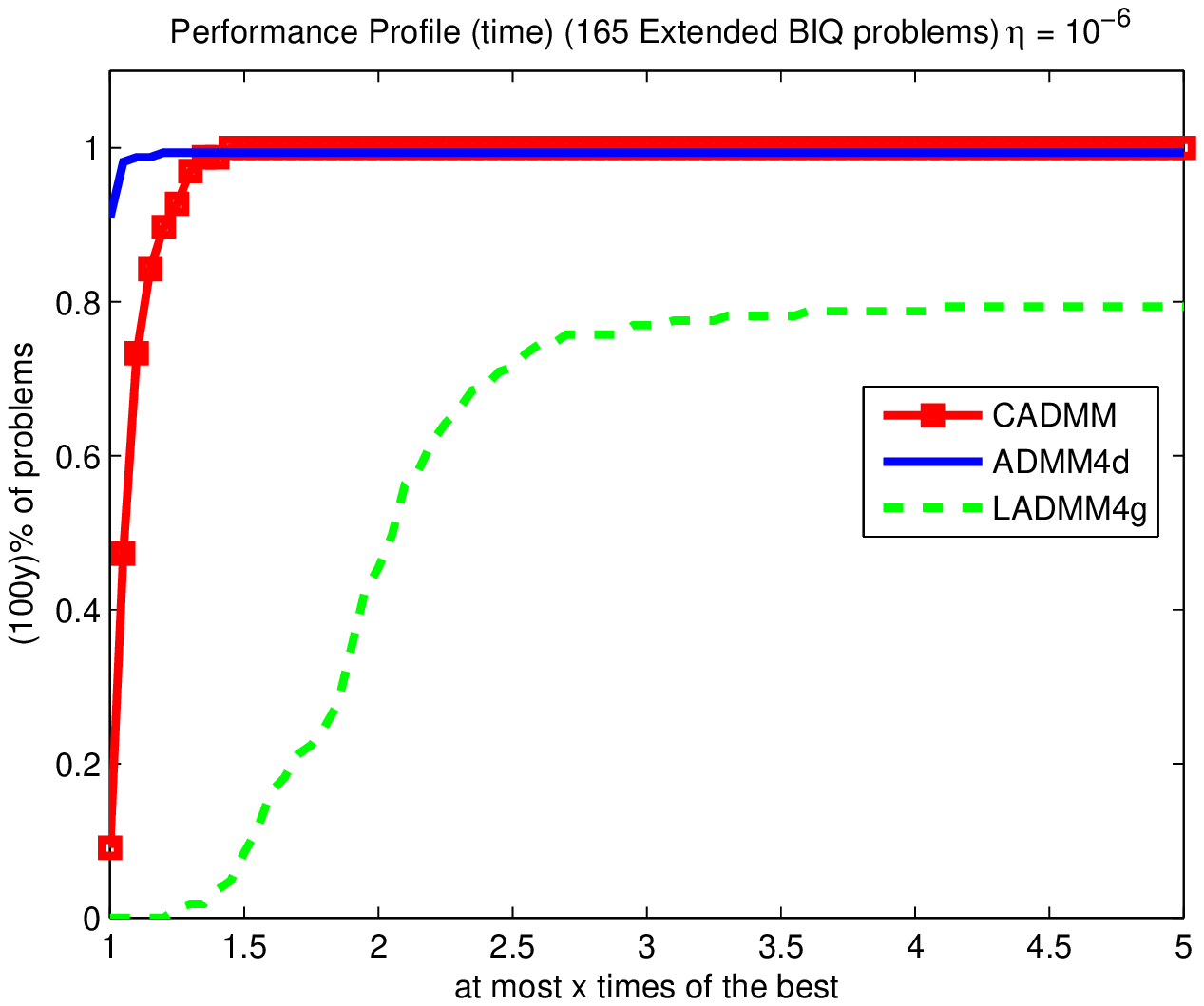}
   \end{minipage}
   \caption{\small Performance profiles of the number of iterations and computing time for EBIQ}
   \label{EBIQ-fig}
  \end{figure}

 \section{Conclusions}

  We have proposed a corrected semi-proximal ADMM by making suitable correction for the directly
  extended semi-proximal ADMM with a large step-size, which does not only have convergent guarantee
  but also enjoys good numerical performance for the general $p$-block $(p\!\ge 3)$ convex
  optimization problems with linear equality constraints. Extensive numerical tests for 
  the doubly nonnegative SDP problems with many linear equality and/or inequality constraints 
  show that the corrected semi-proximal ADMM is superior to the directly extended ADMM with 
  step-size $\tau=1.618$ in terms of the number of iterations,
  and it requires fewer iterations than the latter for $70\%$ test problems within the comparable
  computing time. In particular, for $40\%$ tested problems, its number of iterations is only
  $67\%$ that of the multi-block ADMM with Gaussian back substitution. Thus, the proposed corrected 
  semi-proximal ADMM to a certain extent resolves the dilemma facing all the existing modified 
  versions of the directly extended ADMM. To the best of our knowledge, this is also the first 
  convergent semi-proximal ADMM for the general multi-block convex optimization problem \eqref{prob}.

  \medskip

  We see from the $\tau$ column of Table \ref{tab2}-\ref{tab3} that for most of test instances, 
  the step-size $\tau_k$ of the prediction step computed by our proposed formula lies in the interval 
  $[1.8,1.95]$, while for some test instances (for example, QAP test problems) it will reduce to 
  be strictly less than $1$ when the relative residual $\eta$ is less than a certain threshold.
  It is interesting that the corrected semi-proximal ADMM still yields the desired result, provided that
  the small step-size appears after the relative residual $\eta$ is less than some threshold.
  This phenomenon seems to match well with the linear convergence rate analysis of the multi-block 
  ADMM in \cite{HLuo13}. In our future research work, we will focus on the convergence rate analysis 
  of the corrected semi-proximal ADMM. Another future research work is to explore the effective 
  convergent algorithms for general $p$-block $(p\!\ge 3)$ separable convex 
  optimization based on the directly extended ADMM with large step-size.

 \bigskip
 \noindent
 {\large\bf Acknowledgements.} The authors would like to thank Professor Defeng Sun
 from National University of Singapore for many helpful discussions on the semi-proximal ADMM,
 and Professor Kim-Chuan Toh from National University of Singapore for providing us the code.
 Our thanks also go to Liuqin Yang for some discussions on the implementation of ADMM3c.

  \newpage

 \tiny

 \setlength\tabcolsep{2pt}

\begin{longtable}{|c|c|c|c|c|c|c|}
\caption{\label{tab2} Performance of CADMM, ADMM3d and ADMM3g on $\theta_{+}$, FAP, QAP, BIQ and RCP problems with accuray=$10^{-6}$.
 In this table, $\tau$ column reports the final value of $\tau_k$ of CADMM, and the computing time is in the format of ``hours:minutes:seconds'.
 In the name of the last 120 problems, ``-s, -m, -l'' means ``-small, -medium, -large'', respectively.}\\
\hline
  &   & \multicolumn{1}{c|}{\rm iteration}  &\multicolumn{1}{c|}{$\eta$}    &\multicolumn{1}{c|}{\rm gap}
  &$\tau$ &\multicolumn{1}{c|}{\rm time}\\ {\rm problem} & {$m\mid n$}  &{\rm cadm}$|${\rm adm3d}$|${\rm adm3g}&{\rm cadm}$|${\rm adm3d}$|${\rm adm3g}
  &{\rm cadm}$|${\rm adm3d}$|${\rm adm3g}&{\rm cadm}&{\rm cadm}$|${\rm adm3d}$|${\rm adm3g}\\ \endfirsthead
\caption{(continued) Performance of CADMM, ADMM3d and ADMM3g on $\theta_{+}$, FAP, QAP, BIQ and RCP problems with accuray=$10^{-6}$.
In this table, $\tau$ column reports the final value of $\tau_k$ of CADMM, and the computing time is in the format of ``hours:minutes:seconds'.
In the name of the last 120 problems, ``-s, -m, -l'' means ``-small, -medium, -large'', respectively.}\\
\hline &   & \multicolumn{1}{c|}{\rm iteration}  &\multicolumn{1}{c|}{$\eta$}    &\multicolumn{1}{c|}{\rm gap}
  &$\tau$ &\multicolumn{1}{c|}{\rm time}\\ {\rm problem} & {$m\mid n$}  &{\rm cadm}$|${\rm adm3d}$|${\rm adm3g}&{\rm cadm}$|${\rm adm3d}$|${\rm adm3g}
  &{\rm cadm}$|${\rm adm3d}$|${\rm adm3g}&{\rm cadm}&{\rm cadm}$|${\rm adm3d}$|${\rm adm3g}\\ \endhead
\hline theta4 &  1949 $|$   200 &  311 $|$  319 $|$  485& 9.9-7  $|$  9.7-7  $|$  3.3-7  &  3.6-7  $|$  3.6-7  $|$  -4.9-8  & 1.84 &04 $|$ 04 $|$ 07  \\
 theta42 &  5986 $|$   200 &  186 $|$  193 $|$  222& 9.9-7  $|$  9.9-7  $|$  5.2-7  &  6.1-8  $|$  4.6-8  $|$  9.9-9  & 1.64 &03 $|$ 03 $|$ 03  \\ theta6 &  4375 $|$   300 &  308 $|$  325 $|$  401& 8.6-7  $|$  8.9-7  $|$  4.8-7  &  {\color{green}{ -1.7-6}}  $|$  {\color{green}{ -1.5-6}}  $|$  -7.2-7  & 1.85 &11 $|$ 10 $|$ 14  \\
 theta62 &  13390 $|$   300 &  175 $|$  168 $|$  192& 9.9-7  $|$  9.6-7  $|$  7.9-7  &  7.8-8  $|$  1.4-7  $|$  6.5-8  & 1.59 &06 $|$ 06 $|$ 07  \\ theta8 &  7905 $|$   400 &  306 $|$  334 $|$  404& 9.9-7  $|$  8.7-7  $|$  8.8-7  &  {\color{green}{ 1.6-6}}  $|$  {\color{green}{ -1.3-6}}  $|$  -2.3-7  & 1.84 &19 $|$ 18 $|$ 24  \\
 theta82 &  23872 $|$   400 &  176 $|$  160 $|$  189& 9.9-7  $|$  9.2-7  $|$  4.6-7  &  -1.8-8  $|$  -2.0-7  $|$  9.9-8  & 1.65 &11 $|$ 09 $|$ 13  \\ theta83 &  39862 $|$   400 &  187 $|$  152 $|$  184& 9.9-7  $|$  9.7-7  $|$  7.4-7  &  1.5-9  $|$  7.5-8  $|$  -4.4-9  & 1.92 &13 $|$ 09 $|$ 12  \\
 theta10 &  12470 $|$   500 &  361 $|$  379 $|$  432& 7.4-7  $|$  8.9-7  $|$  9.1-7  &  {\color{green}{ 1.4-6}}  $|$  {\color{green}{ 1.6-6}}  $|$  4.0-7  & 1.84 &35 $|$ 34 $|$ 42  \\
\hline
 \end{longtable}

 \bigskip

 \begin{longtable}{|c|c|c|c|c|c|c|}
 \caption{ \label{tab3} Performance of CADMM, ADMM4d and LADMM4g on the $165$ extended BIQ problems with accuray=$10^{-6}$.
 The computing time is in the format of ``h:m:s'.}\\
 \hline
 &   & \multicolumn{1}{c|}{\rm iteration}  &\multicolumn{1}{c|}{$\eta$}    &\multicolumn{1}{c|}{\rm gap} &$\tau$ &\multicolumn{1}{c|}{\rm time}\\
 {\rm problem} & {$m\mid n$}  &{\rm cadm}$|${\rm adm4d}$|${\rm Ladm4g}&{\rm cadm}$|${\rm adm4d}$|${\rm Ladm4g}
 &{\rm cadm}$|${\rm adm4d}$|${\rm Ladm4g}&{\rm cadm}&{\rm cadm}$|${\rm adm4d}$|${\rm Ladm4g}\\
 \endfirsthead
 \caption{(continued) Performance of CADMM, ADMM4d and LADMM4g on the $165$ extended BIQ problems with accuray=$10^{-6}$.
  The computing time is in the format of ``h:m:s'.}\\
 \hline
  &   & \multicolumn{1}{c|}{\rm iteration}  &\multicolumn{1}{c|}{$\eta$} &\multicolumn{1}{c|}{\rm gap} &$\tau$ &\multicolumn{1}{c|}{\rm time}\\
 {\rm problem} & {$m\mid n$}  &{\rm cadm}$|${\rm adm4d}$|${\rm Ladm4g}&{\rm cadm}$|${\rm adm4d}$|${\rm Ladm4g}
 &{\rm cadm}$|${\rm adm4d}$|${\rm Ladm4g}&{\rm cadm}&{\rm cadm}$|${\rm adm4d}$|${\rm Ladm4g}\\
 \endhead
 \hline
 be100.1 &  101 $|$   101 &  15986 $|$  17066 $|$  28101& 9.9-7  $|$  9.9-7  $|$  9.9-7  &  6.4-8  $|$  2.5-7  $|$  -4.6-7  & 1.83 &1:38 $|$ 1:31 $|$ 3:09  \\
 be100.2 &  101 $|$   101 &  9452 $|$  10803 $|$  21101& 9.9-7  $|$  9.9-7  $|$  9.9-7  &  -6.1-7  $|$  -4.6-7  $|$  -4.5-7  & 1.95 &58 $|$ 57 $|$ 2:21  \\
 be100.3 &  101 $|$   101 &  12768 $|$  13842 $|$  36001& 9.9-7  $|$  9.9-7  $|$  9.9-7  &  -3.0-9  $|$  -1.4-7  $|$  -6.0-7  & 1.95 &1:15 $|$ 1:14 $|$ 4:22  \\ be100.4 &  101 $|$   101 &  14181 $|$  15549 $|$  26641& 9.9-7  $|$  9.9-7  $|$  9.9-7  &  -2.6-8  $|$  -3.5-8  $|$  -1.2-7  & 1.95 &1:24 $|$ 1:23 $|$ 2:59  \\ be100.5 &  101 $|$   101 &  8480 $|$  9899 $|$  14302& 9.9-7  $|$  9.9-7  $|$  9.6-7  &  -5.0-7  $|$  -3.0-7  $|$  -4.7-7  & 1.95 &50 $|$ 52 $|$ 1:36  \\
 be100.6 &  101 $|$   101 &  10803 $|$  12239 $|$  20202& 9.9-7  $|$  9.9-7  $|$  9.9-7  &  -5.0-7  $|$  -7.4-7  $|$  7.0-7  & 1.95 &1:03 $|$ 1:05 $|$ 2:19  \\ be100.7 &  101 $|$   101 &  13926 $|$  15177 $|$  26801& 9.9-7  $|$  9.9-7  $|$  9.9-7  &  -3.4-7  $|$  -5.1-7  $|$  -2.5-7  & 1.90 &1:22 $|$ 1:20 $|$ 3:00  \\ be100.8 &  101 $|$   101 &  12099 $|$  12640 $|$  22301& 9.9-7  $|$  9.9-7  $|$  9.2-7  &  -4.7-7  $|$  -4.5-7  $|$  -2.2-7  & 1.85 &1:11 $|$ 1:07 $|$ 2:31  \\
\hline
 \end{longtable}


\begin{thebibliography}{1}

  \bibitem{BL06}
  {\sc J. M.\ Borwein and A. S.\ Lewis},
  {\em Convex Analysis and Nonlinear Optimization: Theory and Examples}, Springer, 2006.

  \bibitem{Bur09}
  {\sc S.\ Burer},
  {\em On the copositive representation of binary and continuous nonconvex quadratic programs},
  Mathematical Programming, vol. 120, pp. 479-495, 2009.



  \bibitem{CHYY14}
 {\sc C. H.\ Chen, B. S.\ He, Y. Y.\ Ye and X. M.\ Yuan},
 {\em The direct extension of admm for multi-block convex minimization problems is not necessarily convergent},
  Mathematical Programming, Series A, DOI 10.1007/s10107-014-0826-5, 2014.


 \bibitem {DY2012}
 {\sc W.\ Deng and W. T.\ Yin},
 {\em On the global and linear convergence of the generalized alternating direction method of multipliers},
 Rice University CAAM Technical Report TR12-14, 2012.

 \bibitem {DLPY2014}
 {\sc W.\ Deng, M.-J.\ Lai, Z. M.\ Peng and W. T.\ Yin},
 {\em Parallel multi-block ADMM with $o(1/k)$ convergence}, UCLA CAM 13-64, 2014.

 \bibitem{FPST13}
 {\sc M.\ Fazel, T. K.\ Pong, D. F.\ Sun and P.\ Tseng},
 {\em Hankel matrix rank minimization with applications to system identification and realization},
 SIAM Journal on Matrix Analysis, vol. 34, pp. 946-977, 2013.

  \bibitem{Gross11}
  {\sc D.\ Gross},
  {\em Recovering low-rank matrices from few coefficients in any basis},
  IEEE Transactions on Information Theory, vol. 57, pp. 1548-1566, 2011.

  \bibitem {GM75}
 {\sc R.\ Glowinski and A.\ Marrocco},
 {\em Sur l¡¯ approximation par \'{e}l\'{e}ments finis d'ordre un, etla r\'{e}solution, par p\'{e}nalisation-dualit\'{e},
  d'une classe de probl\`{e}mes de dirichlet non lin\'{e}ares},
  Revue Francaise d' Automatique, Informatique et Recherche Op\'{e}rationelle, vol. 9, pp. 41-76, 1975.

 \bibitem {GM76}
 {\sc D.\ Gabay and B.\ Mercier},
 {\em A dual algorithm for the solution of nonlinear variational problems via finite element approximation},
 Computers and Mathematics with Applications, vol. 2, pp. 17-40, 1976.



  \bibitem {Hahn}
 {\sc P.\ Hahn and M.\ Anjos},
 {\em QAPLIB-A Quadratic Assignment Problem Library}.
 http://www.seas.upenn.edu/qaplib.


  \bibitem {HYZC13}
 {\sc D. R.\ Han, X. M.\ Yuan, W. X.\ Zhang and X. J.\ Cai},
 {\em An ADM-based splitting method fo separable convex programming},
 Computational Optimization and Applications, vol. 54, pp. 343-369, 2013.


 \bibitem {HTY12}
 {\sc B. S.\ He, M.\ Tao and X. M.\ Yuan},
 {\em Alternating direction method with gaussian back substitution for separable convex programming},
 SIAM Journal on Optimization, vol. 22, pp. 313-340, 2012.

  \bibitem {HTY14}
 {\sc B. S.\ He, M.\ Tao and X. M.\ Yuan},
 {\em A splitting method for separable convex programming},
 IMA Journal of Numerical Analysis, vol. 22, pp. 1-33, 2014.


 \bibitem {HY13}
 {\sc B. S.\ He and X. M.\ Yuan},
 {\em Linearized alternating direction method of multipliers with Gaussian back substitution for separable
 convex programming}, Numerical Algebra Control Optimization, vol. 3, pp. 247-260, 2013.


  \bibitem {HTXY13}
 {\sc B. S.\ He, M.\ Tao, M. H.\ Xu and X. M.\ Yuan},
 {\em An alternating direction-based contraction method for linearly constrained separable convex programming problems},
  Optimization, vol. 62, pp. 573-596, 2013.

  \bibitem {HLuo13}
 {\sc M. Y.\ Hong and Z. Q.\ Luo},
 {\em On the linear convergence of the alternating direction method of multipliers},
 arXiv preprint arXiv:1208.3922v3, 2013.

  \bibitem {LST14}
  {\sc X. D.\ Li, D. F.\ Sun and K. C.\ Toh},
  {\em A Schur complement based semi-proximal ADMM for convex quadratic programming and extensions},
  accepted by Mathematical Programming, Series A, DOI: 10.1007/s10107-014-0850-5, 2014.


 \bibitem{MPS14}
 {\sc W. M.\ Miao, S. H.\ Pan and D. F.\ Sun},
 {\em A rank-corrected procedure for matrix completion with fixed basis coefficients},
 Submitted to Mathematical Programming (under revision), 2014.


  \bibitem{Negahban11}
  {\sc S.\ Negahban and M. J.\ Wainwright},
  {\em Estimation of (near) low-rank matrices with noise and high-dimensional scaling},
  The Annals of Statistics, vol. 39, pp. 1069-1097, 2011.

  \bibitem{Henstenes76}
  {\sc M. R.\ Hestenes},
 {\em Multiplier and gradient methods},
  Journal of Optimization Theory and Applications, vol. 4, pp. 303-320, 1969.

  \bibitem{Powell69}
  {\sc M.\ Powell},
 {\em A method for nonlinear constraints in minimization problems}, in Optimization,
  R. Fletcher, ed., Academic Press, 1969, pp. 283-298.

 \bibitem{PR09}
 {\sc J.\ Povh and F.\ Rendl},
 {\em Copositive and semidefinite relaxations of the quadratic assignment problem},
 Discrete Optimization, vol. 6, pp. 231-241, 2009.

  \bibitem{PW07}
 {\sc J. M.\ Peng and Y.\ Wei},
 {\em Approximating $k$-means-type clustering via semidefinite programming},
 SIAM Journal on Optimization, 18(2007), pp. 186-205.


  \bibitem{Roc70}
  {\sc R. T.\ Rockafellar},
  {\em Convex Analysis}, 
  Princeton University Press, Princeton, NJ, 1970.

  \bibitem{RW98}
  {\sc R. T.\ Rockafellar and R. J-B.\ Wets},
  {\em Variational Analysis}, Springer, 1998.


 \bibitem{Roc76}
  {\sc R. T.\ Rockafellar},
 {\em Augmented Lagrangians and applications of the proximal point algorithm in convex programming},
 Mathematics of Operations, vol. 1, pp. 97-116, 1976.

  \bibitem {STY14}
 {\sc D. F.\ Sun, K. C.\ Toh and L. Q.\ Yang},
 {\em A convergent proximal alternating direction method of multipliers for conic programming with $4$-block constraints},
  arXiv preprint arXiv:1404.5378, 2014.


 \bibitem {Sloane05}
 {\sc N.\ Sloane},
 {\em Challenge Problems: Independent Sets in Graphs},
 http://www.research. att.com/njas/doc/graphs.html, 2005.

  \bibitem {SSZ08}
 {\sc D. F.\ Sun, J.\ Sun and L. W.\ Zhang},
 {\em  The rate of convergence of the augmented Lagrangian method for nonlinear semidefinite programming},
  Mathematical Programming, vol. 114, pp. 349-391, 2008.


 \bibitem {TYuan11}
 {\sc M.\ Tao and X. M.\ Yuan},
 {\em  Recovering low-rank and sparse components of matrices from incomplete and noisy observations},
  SIAM Journal on Optimization, vol. 21, pp. 57-81, 2011.

 \bibitem{TCCJ92}
 {\sc M.\ Trick, V.\ Chavatal, B.\ Cook, D.\ Johnson, C.\ Mcgeoch and R.\ Tajan},
 {\em The second dimacs implementation challenge-NP hard problems: Maximum clique, graph coloring and satisfiability}.
 http://dimacs.rutgers.edu/Challenges, 1992.

  \bibitem {Toh04}
 {\sc K. C.\ Toh},
 {\em  Solving large scale semidefinite programs via an iterative solver on the augmented systems},
  SIAM Journal on Optimization, vol. 14, pp. 670-698, 2004.


   \bibitem {WGY10}
 {\sc Z. W.\ Wen, D.\ Goldfarb and W. T.\ Yin},
 {\em Alternating direction augmented Lagrangian methods for semidefinite programming},
  Mathematical Programming Computation, vol. 12, , pp. 203-230, 2012.

 \bibitem{WGRPM09}
 {\sc J.\ Wright, A.\ Ganesh, S.\ Rao, Y.\ Peng and Y.\ Ma},
 {\em Robust principle compoenent analysis: Exact recovery of corrupted low-rank matrices by convex optimization},
 in Proceeding of Neural Information Processing Systems, 3(2009).


 \bibitem{WHML13}
 {\sc X. F.\ Wang, M. Y.\ Hong, S. Q.\ Ma and Z. Q.\ Luo},
 {\em Solving multiple-block seprable convex minimization problems using two-block alternating direction
 method of multipliers}, arXiv preprint arXiv: 1308.5294, 2013.



  \bibitem {WXL13}
 {\sc Y.\ Wang, H.\ Xu and C.\ Leng},
 {\em  Provable Subspace Clustering: When LRR meets SSC},
  in NIPS 2013, Lake Tahoe, 2013.

   \bibitem {XW11}
 {\sc M. H.\ Xu and T.\ Wu},
 {\em  A class of linearized proximal alternating direction methods},
  Journal of Optimization Theory and Applications, vol. 151, pp. 321-327, 2011.


 \bibitem {ZJD13}
 {\sc Y. M.\ Zhang, Z. L.\ Jiang and L. S.\ Davis},
 {\em  Learning Structured Low-rank Representations for Image Classification},
  IEEE Conference on Computer Vision and Pattern Recognition, 2013.
 \end{thebibliography}
\end{document}